\documentclass{svjour3}                     
\smartqed  
\usepackage{graphicx}
\usepackage{amssymb, amsmath, enumerate}
\usepackage{hyperref}
\hypersetup{
    colorlinks=true,
    linkcolor=blue,
    filecolor=magenta,      
    urlcolor=cyan,
}
\newcommand{\CC}{\mathbb{C}}

\newcommand{\NN}{\mathbb{N}}

\newcommand{\RR}{\mathbb{R}}

\newcommand{\ZZ}{\mathbb{Z}}

\DeclareMathAlphabet{\mathpzc}{OT1}{pzc}{m}{it}

\newcommand{\cC}{\mathcal{C}}

\newcommand{\cE}{\mathcal{E}}
\newcommand{\cF}{\mathcal{F}}
\newcommand{\cG}{\mathcal{G}}

\newcommand{\cK}{\mathcal{K}}
\newcommand{\cL}{\mathcal{L}}

\newcommand{\cV}{\mathcal{V}}
\newcommand{\cW}{\mathcal{W}}

\newcommand{\pd}{\partial}
\newcommand{\xx}{\mathbf{x}}
\newcommand{\yy}{\mathbf{y}}

\newcommand{\uu}{\mathbf{u}}

\newcommand{\nn}{\mathbf{n}}
\newcommand{\rr}{\mathbf{r}}
\newcommand{\dd}{\mathbf{d}}

\newcommand{\kk}{\mathbf{k}}
\newcommand{\ii}{\mathbf{i}}

\renewcommand{\d}{\mathrm{d}}
\journalname{Numerische Mathematik}
\begin{document}

\title{Corrected Trapezoidal Rules for Boundary Integral Equations in Three Dimensions
}


\author{Bowei Wu        \and
        Per-Gunnar Martinsson 
}


\institute{B. Wu \at
              Oden Institute for Computational Engineering and Sciences, University of Texas at Austin, Austin, TX 78712, USA\\
              \email{boweiwu@utexas.edu}           
           \and
           P.G. Martinsson \at
           Department of Mathematics, University of Texas at Austin, Austin, TX 78712, USA\\
              \email{pgm@oden.utexas.edu}       
}

\date{Received: date / Accepted: date}

\maketitle

\begin{abstract}
The manuscript describes a quadrature rule that is designed for the high order discretization of boundary integral equations (BIEs) using the Nystr\"{o}m method. 
The technique is designed for surfaces that can naturally be parameterized using a uniform grid on a rectangle, such as deformed tori, or channels with periodic boundary conditions.
When a BIE on such a geometry is discretized using the Nystr\"{o}m method based on the Trapezoidal quadrature rule, the resulting scheme tends to converge only slowly, due to the singularity in the kernel function.
The key finding of the manuscript is that the convergence order can be greatly improved by modifying only a very small number of elements in the coefficient matrix.
Specifically, it is demonstrated that by correcting only the diagonal entries in the coefficient matrix, $O(h^{3})$  convergence can be attained for the single and double layer potentials associated with both the Laplace and the Helmholtz kernels. 
A nine-point correction stencil leads to an $O(h^5)$ scheme.
The method proposed can be viewed as a generalization of the quadrature rule of Duan and Rokhlin, which was designed for the 2D Lippmann-Schwinger equation in the plane.
The techniques proposed are supported by a rigorous error analysis that relies on Wigner-type limits involving the Epstein zeta function and its parametric derivatives.
\keywords{Wigner limit \and Epstein zeta function \and Singular quadrature \and Boundary integral equations}
\subclass{65R20 \and 65D32 \and 45B05} 
\end{abstract}

\section{Introduction}

\subsection{Nystr\"om discretization and weakly singular kernels}
Boundary value problems associated with linear elliptic partial differential equations, such as the Laplace and Helmholtz equations, are often reformulated as boundary integral equations (BIE) of the second kind. For concreteness, let us consider a BIE of the form
\begin{equation}
\label{eq:BIE}
\sigma(\xx) + \int_\Gamma \cK(\xx,\yy) \sigma(\yy)\d S(\yy) = f(\xx),\quad \xx\in\Gamma
\end{equation}
where $\Gamma\subset\RR^3$ is a smooth surface, where $\d S$ is the surface element, and where the kernel $\cK\in L^2(\Gamma\times\Gamma)$ is weakly singular as $|\xx-\yy|\rightarrow 0$. We focus on the situation where the given data function $f$ is smooth, in which case the solution $\sigma$ is smooth as well. 

We discretize (\ref{eq:BIE}) using a Nystr\"om discretization \cite[\S12.2]{kress2014linear} based on a quadrature rule with nodes $\{\xx_{i}\}_{i=1}^{N} \subset \Gamma$ and weights $\{w_{i}\}_{i=1}^{N}$ for which the approximation
\begin{equation}
\label{eq:smoothquad}    
\int_{\Gamma}\varphi(\xx)\d S(\xx) \approx \sum_{i=1}^{N} w_{i}\,\varphi(\xx_{i})
\end{equation}
is accurate for \textit{smooth} functions $\varphi$. In the Nystr\"om method, we first collocate (\ref{eq:BIE}) to the quadrature nodes $\{\xx_{i}\}$, and then replace the continuum integral by a quadrature supported on the same nodes, to obtain the linear system
\begin{equation}
\label{eq:nystromlinsys}
\sigma(\xx_i) + \sum_{j=1}^N \mathbf{K}(i,j)\,\sigma(\xx_j) = f(\xx_i),\qquad  i = 1,\,2,\,\dots,\,N.
\end{equation}
For this to work, we need to build an $N\times N$ coefficient matrix $\mathbf{K}$ such that
\begin{equation}
\label{eq:nystromquad}    
\sum_{j=1}^{N}\mathbf{K}(i,j)\,\sigma(\xx_{j})\approx 
\int_{\Gamma}\cK(\xx_{i},\,\yy)\sigma(\yy)\,\d S(\yy) ,
\qquad i = 1,\,2,\,\dots,\,N
\end{equation}
holds to high accuracy. If the kernel $\cK$ were smooth, then this task would be easy since we could then use the standard quadrature rule  (\ref{eq:smoothquad}) and simply set
\begin{equation}
\label{eq:niceK}    
\mathbf{K}(i,j) = \cK(\xx_{i},\xx_{j})\,w_{j}.
\end{equation}
The challenging question in this context is how to find a matrix $\mathbf{K}$ for which the approximation (\ref{eq:nystromquad}) holds to high accuracy despite the singularity in the kernel. In this manuscript, we address it for the specific case where the surface $\Gamma$ is parameterized over a rectangle $R \subset \mathbb{R}^{2}$, the nodes $\{\xx_{i}\}$ are the images of a uniform grid on $R$, and the base quadrature (\ref{eq:smoothquad}) is the Trapezoidal rule on $R$.

A particular benefit of the method that we present is that the basic relation (\ref{eq:niceK}) holds for \textit{almost} all entries of $\mathbf{K}$. This makes the technique proposed attractive to use in combination with fast solvers such as those based on the Fast Multipole Method \cite{rokhlin1987} or fast direct solvers \cite{2019_martinsson_book}.

\subsection{Prior work}
Numerical evaluation of boundary integral operators, or integration of singular functions in general, is a rich topic with a long history, see for example \cite{atkinson1997numerical,davis2007methods,kress2014linear} for more detailed reviews. Broadly speaking, singular quadrature
techniques can be classified into three categories:
\begin{enumerate}
\item \emph{Singularity cancellation methods} typically apply a change of integration variables (e.g., from Cartesian to polar coordinates) such that the singularity is (fully or partially) cancelled by the Jacobian, then the resulting smoother integrand is integrated using a regular quadrature. Examples in this category include \cite{duffy1982quadrature,atkinson2004quadrature,bruno2001fast,bremer2012nystrom}. A related method is the singularity regularization approach, e.g. \cite{beale2001convergent}, where the singularity is ``smoothed out'' locally in a sophisticated way such that, when integrated using a regular quadrature, the regularization error is balanced with the discretization error.
\item \emph{Singularity subtraction methods} proceed by first subtracting the singular component from the integrand and integrating the remaining smooth component with a regular quadrature, then the singular component is integrated analytically and added back to the final result. Examples in this category include \cite{merkel1986integral,helsing2008evaluation,helsing2013higher}. 
\item \emph{Singularity correction methods} also split the integrand into regular and singular components; while the regular component is handled by the underlying regular quadrature, the singular component is now integrated numerically (instead of analytically) using a specially designed quadrature. The outcome of this approach is a \emph{modified} quadrature, where the original regular quadrature weights are modified to accommodate the singularity. For example, extrapolation methods such as \cite{haroldsen1998numerical} or product quadrature methods such as \cite{kress1991boundary} modify weights \emph{globally}.
\end{enumerate}

Our focus in this manuscript is on uniform discretizations in parameter space, and quadratures based on the classical Trapezoidal rule. For smooth integrands, the trapezoidal rule converges super algebraically fast, and the existing error analysis is very precise \cite{trefethen2014exponentially}. We observe that in the context of BIEs on toroidal domains, the integrand is periodic, so the only loss of smoothness is due to the singularity in the kernel. (Corrections for edge effects are described in \cite{alpert1995high}.)

In order to analyze the behavior of the Trapezoidal rule for functions with isolated singularities, it is common to introduce a ``punctured trapezoidal rule'' that omits the singular point. By itself, such a rule would be of low accuracy, but a number of strategies for improving its convergence speed have been developed.
For example, when the integration domain coincides with the Euclidean space ($\Gamma=\RR^n$), high-order and robust corrected trapezoidal rules are developed for $\log|x|$ and $|x|^{-s}, -1<s<1$, in $\RR^1$ \cite{kapur1997high}, for $\log|\xx|, |\xx|^{-1},$ and $H_0^{(1)}(\kappa|\xx|)$ in $\RR^2$ \cite{aguilar2002high,marin2014corrected,duan2009high}, and for $|\xx|^{-1}$ in $\RR^3$ \cite{aguilar2005high}. When the integration domain is on a curved surface ($\Gamma\subsetneq\RR^n$), several powerful techniques exist for the case $n=2$. For example, Alpert in \cite{alpert1999hybrid} developed a hybrid Gaussian-trapezoidal quadrature that introduced correction weights on a local auxiliary grid that is unevenly-spaced; for on-grid corrections, Kapur and Rokhlin in \cite{kapur1997high} have developed quadratures for a variety of singularities that are not directly elementary functions, the correction weights of these quadratures grow rapidly as the order of correction increases, hence are less stable when the correction order is six or higher, see \cite{hao2014high}; Alpert alleviated this issue by using extra correction weights and minimizing their sum of squares \cite{alpert1995high}. Most of these quadratures are designed based on the error analysis of the punctured trapezoidal rule.
We are not aware of any analogous techniques that work on curved surfaces in $\RR^3$.

\subsection{Contributions of the present work}

We present a systematic approach to constructing locally corrected trapezoidal rules for the Laplace and Helmholtz layer potentials on smooth surfaces in $\RR^3$. The key innovation of our method is that we connect the error analysis of the punctured trapezoidal rule to lattice sum theory and the Epstein zeta function (a generalization of the Riemann zeta function), and based on them develop fast algorithms that efficiently compute the limiting error coefficients as the grid spacing $h\to0$. From there, local correction weights are calculated based on the idea of moment fitting \cite{wissmann1986partially}. We prove the convergence of the correction weights in Theorem \ref{thm:converged_wigner-type} and justify their use for the intermediate $h$ grid by Theorem \ref{thm:high-order-with-converged-wigner}. Our method generalizes straightforwardly to other on-surface integral operators such as the Stokes potentials, which share the same $|\xx|^{-1}$ type singularity as the Laplace potentials.

We mention that the connection between the error of the punctured trapezoidal rule and the zeta function was first mentioned by Marin et.\ al.\ in \cite[Lemma A2]{marin2014corrected} for the $|x|^{-s}$ function, $0<s<1$, on $\RR^1$, which enabled a complete error analysis of their 1D quadrature. 
In this manuscript, we have generalized this zeta function connection to higher dimensions and consequently complete the convergence analysis of \cite{marin2014corrected} for their 2D quadrature. 
In addition, as is pointed out in \cite{marin2014corrected}, many existing locally corrected trapezoidal methods in $\RR^2$ and $\RR^3$ are derived heuristically and lack complete convergence analysis, which in turn is due to the lack of expressions for the converged correction weights (as $h\to0$). The generalized zeta function connection shown in this manuscript provides a promising tool to overcome these difficulties and serves as an analytical foundation for many related singular trapezoidal methods.

\subsection{Organization}
Section \ref{sc:slp_wigner} introduces the concept of a Wigner limit, and shows how it is connected to the trapezoidal-rule discretization of the Laplace single-layer potential. 
Section \ref{sc:eval_epstein_zeta} presents an efficient algorithm to compute the correction weights for the trapezoidal rule based on the Epstein zeta function.
Section \ref{sc:dlp_epstein_der} further extends the method by making the connection between the parametric derivatives of the Epstein zeta function and the correction weights associated with the Laplace double-layer potential. Higher-order corrected trapezoidal rules are developed and analyzed in Section \ref{sc:high-order_quad}; in particular, Section \ref{sc:high-order:converged_Wigner-type} presents the complete analysis for the zeta function connection. Numerical results are presented in Section \ref{sc:results}, where we implement the $O(h^5)$ corrected trapezoidal rules and demonstrate the scalability of our method by combining the corrected quadratures with the Fast Multipole Method (FMM) to solve the Laplace and Helmholtz boundary value problems. Finally, we draw conclusions and point out future directions.

\section{The Wigner limit and correction of the single-layer potential}
\label{sc:slp_wigner}

In this section, we define the Wigner limit and demonstrate its connection to the convergence analysis for the punctured trapezoidal rule when applied to a single layer potential on a curved surface $\Gamma$. We assume that $\Gamma$ is smooth, and that it is locally parameterized by $\rr: (u,v)\mapsto \rr(u,v)\in\Gamma$, for $u,v \in [-a,a]$. Without loss of accuracy, we assume that $\rr(0,0) = \mathbf{0}$. We set $r=|\rr|$, and let $\rr_u(u,v)$ and $\rr_v(u,v)$ denote the tangent vectors at $(u,v)$. $J(u,v)=|\rr_u\times\rr_v|$ is the Jacobian. 

Let $\sigma$ denote a smooth function on $\Gamma$. To keep the notation uncluttered, we view $\sigma$ as a function over $[-a,a]^{2}$ so that $\sigma = \sigma(u,v)$. Now consider the following punctured trapezoidal approximation of the Laplace single-layer potential (SLP) on an $h$-grid and with center correction
\begin{equation}
\begin{aligned}
&\int_{-a}^a\int_{-a}^a\frac{1}{4\pi r}\sigma(u,v)J(u,v)\,\d u\d v  \\
&\approx  \frac{1}{4\pi}\sum_{i=-N}^{N}\sideset{}'\sum_{j=-N}^{N}\frac{\sigma(ih,jh)J(ih,jh)}{r(ih,jh)} + \frac{\sigma(0,0)J(0,0)}{4\pi}\tau_{0,0}
\end{aligned}
\label{eq:lap_corr}
\end{equation}
where $a = (N+\frac{1}{2})h$ and the prime $'$ indicates that $(i,j)\neq(0,0)$. We assume that $\sigma$ is compactly supported in the integration domain to avoid the discussion of boundary corrections for the trapezoidal rule. (In practical applications, $\sigma$ will typically be periodic, rather than compactly supported, but this makes no difference in the analysis.) Our goal is to determine the correction weight $\tau_{0,0}$ such that the  corrected trapezoidal rule is $O(h^3)$ accurate whenever $\sigma$ and $\Gamma$ are sufficiently smooth.

One can show that $\tau_{0,0} = \cC_0\cdot h$, where $\cC_0$ is a constant that depends on the singularity $1/r$ and the local geometry (we defer the analysis to Section \ref{sc:high-order_quad}). To determine $\cC_0$, we consider an approximation
that is valid for $u$ and $v$ small of the form
\begin{equation}
\frac{1}{r}\approx\frac{1}{\sqrt{Q_A(u,v)}}
\end{equation}
where
\begin{equation}
Q_A(u,v)=Eu^2+2Fuv+Gv^2,\quad A := \begin{bmatrix} E & F\\ F & G\end{bmatrix}
\end{equation}
is the \emph{first fundamental form} at $(0,0)$ with coefficients
$$E=\rr_u(0,0)\cdot\rr_u(0,0),\quad F=\rr_u(0,0)\cdot\rr_v(0,0),\quad G=\rr_v(0,0)\cdot\rr_v(0,0).$$
Substituting this approximation into \eqref{eq:lap_corr}, with the smooth part $\sigma(u,v)J(u,v)\equiv1$ and changing integration variables to $u\mapsto h\,u, v\mapsto h\,v$, gives
\begin{equation}
\begin{aligned}
\tau_{0,0} \approx (-W_A^{(N)}(1))\cdot h
\end{aligned}
\end{equation}
where
\begin{align}
W_A^{(N)}(s) &:= Z_A^{(N)}(s) - I_A^{(N)}(s) \label{eq:wigner_N} \\
Z_A^{(N)}(s) &:= \sum_{i=-N}^{N}\sideset{}'\sum_{j=-N}^{N} Q_{A}(i,j)^{-s/2}\\
I_A^{(N)}(s) &:= \int_{-N-\frac{1}{2}}^{N+\frac{1}{2}}\int_{-N-\frac{1}{2}}^{N+\frac{1}{2}}Q_{A}(u,v)^{-s/2}\d u\d v
\end{align}
It is expected that
\begin{equation}
    \tau_{0,0} = -h\cdot \lim_{N\to\infty}W_A^{(N)}(1) =: (-W_A(1))\cdot h.
\end{equation}

The limit $W_A(s)$ is called the \emph{Wigner limit} which appears in the computation of electron sums in chemistry (see \cite[Sec. 7.1]{borwein2013lattice} and \cite{borwein2014lattice}). 

\begin{theorem}
\label{thm:wigner_lim}
For any positive definite quadratic form $Q_A$, the Wigner limit $W_A(s)$ exists in the strip $0 < \mathrm{Re}\, s < 2$ and coincides therein with the analytic continuation of $Z_A(s) :=\lim_{N\to\infty}Z_A^{(N)}(s)$. In particular, $W_A(1) = Z_A(1)$.
\end{theorem}

Theorem \ref{thm:wigner_lim} was proved in \cite[Theorem 1]{borwein1989analysis}. Notice that the limits of $Z_A^{(N)}(s)$ and $I_A^{(N)}(s)$ only exist when $\mathrm{Re}\,s>2$, and it is remarkable that the integrals $I_N(s)$ play no role in the Wigner limit when $\mathrm{Re}\,s<2$. Consequently, the constant $\cC_0$ is now given by
\begin{equation}
\cC_0 = -W_A(1) \equiv -Z_A(1),
\end{equation}
so it reduces to finding the analytic continuation of $Z_A(s)$ at $s=1$.

\begin{remark}[Epstein zeta function]
The (analytically continued) function $Z_A(s)$ is the two-dimensional \emph{Epstein zeta function} \cite{epstein1903theorie,epstein1906theorie}. It is analytic everywhere in $\CC$ except a simple pole at $s = 2$. This is a generalization of the one-dimensional Riemann zeta function, $\zeta(s)$, which has a simple pole at $s=1$.
\end{remark}

\begin{remark}
For certain quadratic forms $Q_A$, the limits $Z_A(s)$ are known exactly (e.g., see \cite[Table 1.6]{borwein2013lattice}). For examples:
\begin{itemize}
\item when $E=G=1$ and $F=0$, i.e., $A=I$,
\begin{equation}
Z_I(s) = \sideset{}'\sum_{i,j}\frac{1}{(i^2+j^2)^{s/2}} = 4\,\zeta\left(\frac{s}{2}\right)\beta\left(\frac{s}{2}\right)
\end{equation}
where $\beta(s)=L_{-4}(s)$ is the Dirichlet Beta function (or a Dirichlet-L function modulo $4$). Consequently, the correction weight for the Laplace SLP at a point parameterized by a locally conformal map ($E=G$ and $F=0$) is 
$$\tau_{0,0} = \frac{-4\zeta\left(\frac{1}{2}\right)\beta\left(\frac{1}{2}\right)}{\sqrt{E}}h\approx 3.900264920001956\frac{h}{\sqrt{E}}.$$
The number $3.90026\dots$ has appeared in \cite[Table 2]{marin2014corrected} as expected.
\item when $E = 1, F = \frac{1}{2}, G = 1$,
\begin{equation}
Z_A(s) = \sideset{}'\sum_{i,j}\frac{1}{(i^2+ij+j^2)^{s/2}} = 6\,\zeta\left(\frac{s}{2}\right)L_{-3}\left(\frac{s}{2}\right)
\end{equation}
where $L_{-3}(s)=1-2^{-s}+4^{-s}-5^{-s}+7^{-s}+\dots$ is a Dirichlet L-function modulo $3$. Consequently, the correction weight for $E = G = 2F$ (i.e. $\rr_u$ and $\rr_v$ form a $\frac{\pi}{3}$ angle) is
$$\tau_{0,0} = \frac{-6\zeta\left(\frac{1}{2}\right)L_{-3}\left(\frac{1}{2}\right)}{\sqrt{E}}h\approx 4.213422636136907 \frac{h}{\sqrt{E}}.$$
\end{itemize}
\end{remark}

\begin{remark}[Helmholtz SLP]\label{rmk:helmholtz_slp}
The correction weight for the Helmholtz SLP with wavenumber $\kappa$, denoted $\tau_{0,0}^\kappa$, is related to the Laplace weight $\tau_{0,0}$ by
\begin{equation}
\tau_{0,0}^\kappa = \tau_{0,0}+i\kappa\cdot h^2 = (- Z_A(1)+i\kappa h)\cdot h
\end{equation}
This connection can be shown by expanding the Helmholtz SLP kernel in $r$,
\begin{equation}
\frac{e^{i\kappa r}}{r} = \frac{1}{r}+ i\kappa - \frac{\kappa^2}{2}r + O(r^2),
\end{equation}
such that $\tau_{0,0}$ handles the $\frac{1}{r}$ singularity and the weight $i\kappa\cdot h^2$ handles the constant term $i\kappa$.
\end{remark}

\begin{remark}[different grid spacings]
When the grid spacings are $h_1$ in the $u$-direction and $h_2$ in the $v$-direction, a similar derivation gives the Laplace SLP correction weight
\begin{equation}
\tau_{0,0} = (-Z_{\tilde{A}}(1))\cdot \sqrt{h_1h_2},\quad \tilde A :=\begin{bmatrix}E\cdot\frac{h_1}{h_2} & F\\ F & G\cdot\frac{h_2}{h_1}\end{bmatrix},
\end{equation}
and the corresponding Helmholtz correction weight becomes
\begin{equation}
\tau_{0,0}^\kappa = (-Z_{\tilde{A}}(1)+i\kappa \sqrt{h_1h_2})\cdot \sqrt{h_1h_2}.
\end{equation}
\end{remark}

\section{Evaluation of the 2D Epstein zeta function \texorpdfstring{$Z_A(s)$}{TEXT}}
\label{sc:eval_epstein_zeta}

Computing the Wigner limit $W_A(s)$ using the definition \eqref{eq:wigner_N} as $N\to\infty$ would result in cancellation errors since both $Z_A^{(N)}(s)$ and $I_A^{(N)}(s)$ are divergent for $s<2$. Fortunately, Theorem \ref{thm:wigner_lim} allows us to avoid the divergent quantities by directly evaluating the Epstein zeta function $Z_A(s)$ so that $W_A(s)=Z_A(s)$.

A fast algorithm to evaluate the Epstein zeta function in general is proposed by Crandall in the note \cite{crandall1998fast}. In this section, we describe a simplified algorithm tailored to our interests, namely, in two dimensions and for the particular value at $s=1$. The full formula for any $s$ is provided in Appendix \ref{sc:app:epstein_higher_deriv}.

The 2D Epstein zeta function for $\mathrm{Re}\,s>2$ is defined as
\begin{equation}
Z_A(s) = \sideset{}{'}\sum_{i,j}Q_A(i,j)^{-s/2} = \sideset{}{'}\sum_{i,j}\frac{1}{(Ei^2+2Fij+Gj^2)^{s/2}}
\label{eq:epstein_zeta_summation_form}
\end{equation}
for any positive definite quadratic form $Q_A$. The analytic continuation of $Z_A(s)$ to the whole complex plane (except a simple pole at $s=2$) is given by the following integral representation \cite[Eq.(1.2.8),(1.2.11)]{borwein2013lattice}:
\begin{equation}
\begin{aligned}
\pi^{-\frac{s}{2}}\Gamma\left(\frac{s}{2}\right)Z_A(s) = \frac{2}{(s-2)\sqrt{D}} - \frac{2}{s} + \int_1^\infty\mathrm{d}t\cdot t^{s/2-1}\sum_{i,j}{}'e^{-\pi Q_A(i,j) t}\\
+\frac{1}{\sqrt{D}}\int_1^\infty\mathrm{d}t\cdot t^{(2-s)/2-1}\sum_{i,j}{}'e^{-\pi \overline{Q_A}(i,j) t}
\end{aligned}
\label{eq:epstein_anal}
\end{equation}
where $D := \det A = EG-F^2$ and $\overline{Q_A}(i,j) := Q_{A^{-1}}(i,j) =  (Gi^2-2Fij+Ej^2)/D$ is the quadratic form associated with $A^{-1}$; when $s=0$, $Z_A(0) = \lim_{s\to0} Z_A(s) = -1$ for any $Q_A$.

Before using \eqref{eq:epstein_anal} to calculate $Z_A(1)$, we make the following simplifications:
\begin{enumerate}[1)]
\item Assume that the determinant $D=1$. Otherwise one can first scale $E,F,G$ by $1/\sqrt{D}$ to obtain a quadratic form 
$$\tilde Q_A(i,j):=Q_{A/\sqrt{D}}(i,j)=\frac{Q_A(i,j)}{\sqrt{D}}$$
with unit determinant, and then scale back the final result $Z_{A/\sqrt{D}}(s)$ by $\sqrt{D}^{-s/2}$.
\item Notice that by symmetry $$\sum_{i,j}{}'e^{-\pi \overline{Q_A}(i,j) t} = \sum_{i,j}{}'e^{-\pi \overline{Q_A}(-j,i) t} = \sum_{i,j}{}'e^{-\pi Q_A(i,j)\,t/D},$$
and in the case of unit determinant, this further simplifies to 
$$\sum_{i,j}{}'e^{-\pi \overline{\tilde Q_A}(i,j) t} = \sum_{i,j}{}'e^{-\pi \tilde{Q}_A(i,j) t}$$
\end{enumerate}
Then the formula \eqref{eq:epstein_anal} simplifies to:
\begin{equation}
Z_A(1) = \frac{1}{\sqrt[4]{D}}Z_{A/\sqrt{D}}(1) = \frac{1}{\sqrt[4]{D}}\left(-4+2\sum_{i,j}{}' \frac{\Gamma(\frac{1}{2}, \pi \tilde Q_A(i,j))}{\sqrt{\pi \tilde Q_A(i,j)}}\right).
\label{eq:epstein_anal_s1}
\end{equation}
where the summands are incomplete Gamma function values $\Gamma(\frac{1}{2},x)$ scaled by $\frac{1}{\sqrt{x}}$, which is positive and bounded:
$$\frac{1}{\sqrt{x}}\Gamma\left(\frac{1}{2}, x\right) = \frac{1}{\sqrt{x}}\int_{x}^\infty t^{\frac{1}{2}-1}\,e^{-t}\,\mathrm{d}t < \frac{e^{-x}}{x}.$$

Since the upper bound $e^{-x}/x$ is strictly decreasing and less than the double-precision machine epsilon $\epsilon_\text{machine}$ for all $x\geq 33$, we only need to sum over all $i,j$ such that $\pi \tilde Q_A(i,j) < 33$. To this end, note that
$$\tilde Q_A(i,j)\geq\lambda \cdot(i^2+j^2)$$
where $\lambda$ is the smaller eigenvalue of $\tilde Q_A$ given by
$$\lambda = \left(\frac{E+G}{2}-\frac{\sqrt{(E-G)^2+4F^2}}{2}\right)/\sqrt{D}>0,$$
therefore the sum \eqref{eq:epstein_anal_s1} can be truncated to $-N\leq i,j\leq N$ where $N = \max\{n:\pi\lambda n^2<33\}$, this results in the following practical evaluation formula:
\begin{equation}
Z_A(1) \approx \frac{1}{\sqrt[4]{D}}\left(-4+2\sum_{i=-N}^N\sum_{j=-N}^N{}' \frac{\Gamma(\frac{1}{2}, \pi \tilde Q_A(i,j))}{\sqrt{\pi \tilde Q_A(i,j)}}\right),\quad N = \Bigg\lfloor \sqrt{\frac{33}{\pi\lambda_-}}\Bigg\rfloor
\label{eq:epstein_anal_s1_trunc}
\end{equation} 
which guarantees to give double-precision accuracy. In practice (such as for the examples in Section \ref{sc:results}), we find that generally $N\lesssim 5$. Furthermore, the fact that $\tilde Q_A(i,j) = \tilde Q_A(-i,-j)$ can save the total work by another factor of $2$ when evaluating the above sum. Therefore the evaluation formula \eqref{eq:epstein_anal_s1_trunc} is highly efficient.

\begin{remark}
The eigenvalue $\lambda$ is in fact the singular value of the local basis $(\rr_u,\rr_v)$, therefore using a ``good'' parameterization such that the local bases are close to orthonormal everywhere will reduce the cost of evaluating $Z_A(s)$. However, the correction weights only need to be pre-computed once, and we find that in practice this cost is typically small compared to other costs such as the iterative application of the punctured trapezoidal rule via fast summation techniques.
\end{remark}

\section{Double-layer potential and parametric derivatives of the Epstein zeta function} 
\label{sc:dlp_epstein_der}
We now turn to the derivation of the correction for the Laplace double-layer potential (DLP). Assuming $\rr(0,0)=\mathbf{0}$ such that $\rr(0,0)-\rr(u,v)\equiv-\rr(u,v)$, we look for a correction formula of the form
\begin{equation}
\begin{aligned}
&\int_{-a}^a\int_{-a}^a\frac{-\rr(u,v)\cdot\nn(u,v)}{4\pi r^3}\sigma(u,v)J(u,v)\d u\d v \\ &\approx \frac{1}{4\pi}\sum_{i=-N}^{N}\sideset{}'\sum_{j=-N}^{N}\frac{-\rr(ih,jh)\cdot\nn(ih,jh)}{r(ih,jh)^3}\sigma(ih,jh)J(ih,jh) + \frac{\sigma(0,0)J(0,0)}{4\pi}\tau_{0,0}
\end{aligned}
\label{eq:lap_dlp_corr}
\end{equation}
where $\nn = \frac{\rr_u\times\rr_v}{|\rr_u\times\rr_v|}$ is the surface unit normal at $\rr(u,v)$.

As with the SLP kernel, we first approximate the DLP kernel by (again, analysis deferred to Section \ref{sc:high-order_quad})
\begin{equation}
\frac{-\rr\cdot\nn}{r^3}\approx \frac{Q_B(u,v)/2}{\sqrt{Q_A(u,v)}^3}
\end{equation}
for $u,v\approx 0$, where
\begin{equation}
Q_B(u,v) = \cE u^2+2\cF uv+\cG v^2,\quad B := \begin{bmatrix} \cE & \cF\\ \cF & \cG\end{bmatrix}
\end{equation}
is the \emph{second fundamental form} at $(0,0)$ with coefficients
$$\cE=\rr_{uu}(0,0)\cdot\nn(0,0),\quad \cF=\rr_{uv}(0,0)\cdot\nn(0,0),\quad \cG=\rr_{vv}(0,0)\cdot\nn(0,0).$$
Substituting the approximation into \eqref{eq:lap_dlp_corr} with $\sigma(u,v)J(u,v)\equiv1$ gives
\begin{equation}
\begin{aligned}
\tau_{0,0} &\approx (-W_{A,B}^{(N)}(1))\cdot h
\end{aligned}
\label{eq:wigner_limit_dlp}
\end{equation}
where
\begin{align}
W_{A,B}^{(N)}(s) &:= Z_{A,B}^{(N)}(s) - I_{A,B}^{(N)}(s)\\
Z_{A,B}^{(N)}(s) &:= \sum_{i=-N}^{N}\sideset{}'\sum_{j=-N}^{N}\frac{Q_B(i,j)\cdot s/2}{Q_A(i,j)^{s/2+1}}\\
I_{A,B}^{(N)}(s) &:= \int_{-N-\frac{1}{2}}^{N+\frac{1}{2}}\int_{-N-\frac{1}{2}}^{N+\frac{1}{2}}\frac{Q_B(u,v)\cdot s/2}{Q_A(u,v)^{s/2+1}}\d u\d v.
\end{align}
Then analogous to Theorem \ref{thm:wigner_lim}, we have:
\begin{theorem}
\label{thm:wigner_lim_deriv}
The Wigner-type limit $W_{A,B}(s):=\lim_{N\to\infty}W_{A,B}^{(N)}(s)$ exists in the strip $0<\mathrm{Re}\,s<2$ and coincides therein with the analytic continuation of  $Z_{A,B}(s):=\lim_{N\to\infty}Z_{A,B}^{(N)}(s)$.
\end{theorem}
\begin{proof}
This is a corollary of Theorem \ref{thm:converged_wigner-type} in Section \ref{sc:high-order:converged_Wigner-type}.
\qed\end{proof}

Theorem \ref{thm:wigner_lim_deriv} implies that
the exact value of $\tau_{0,0}$ in \eqref{eq:lap_dlp_corr} is
\begin{equation}
    \tau_{0,0} = (-W_{A,B}(1))\cdot h \equiv (-Z_{A,B}(1))\cdot h,
\end{equation}
where, once again, the integrals $I_{A,B}^{(N)}(s)$ do \emph{not} play any role in this limiting value. The problem of computing the conditionally convergent $W_{A,B}(s)$ is again reduced to one of evaluating $Z_{A,B}(s)$. 

To efficiently evaluate $Z_{A,B}$, we make the following observation:
\begin{equation}
\frac{(\cE u^2+2\cF uv+\cG v^2)\cdot s/2}{(Eu^2+2Fuv+Gv^2)^{s/2+1}} = -\left(\cE\frac{\partial}{\partial E}+\cF\frac{\partial}{\partial F}+\cG\frac{\partial}{\partial G}\right)\frac{1}{(Eu^2+2Fuv+Gv^2)^{s/2}}
\end{equation}
which implies that differentiating the Epstein zeta function \eqref{eq:epstein_zeta_summation_form} gives
\begin{equation}
Z_{A,B}(s) = -\left(\cE\frac{\partial}{\partial E}+\cF\frac{\partial}{\partial F}+\cG\frac{\partial}{\partial G}\right) Z_A(s).
\end{equation}

Therefore, an analytic formula for $Z_{A,B}(1)$ can be constructed by substituting the above relation into \eqref{eq:epstein_anal_s1}, which yields
\begin{equation}
\begin{aligned}
Z_{A,B}(1)=\frac{1}{\sqrt[4]{D}}\Bigg(-2H_{A,B}+&\sum_{i,j}{}'\Bigg\{\frac{\tilde Q_B}{\tilde Q_A}\cdot\frac{\Gamma(\frac{1}{2},\pi\tilde Q_A)}{\sqrt{\pi\tilde Q_A}}\\
&\qquad +\left(2\frac{\tilde Q_B}{\tilde Q_A}-H_{A,B}\right)\,e^{-\pi\tilde Q_A}\Bigg\}\Bigg)
\end{aligned}
\label{eq:epstein_anal_s1_deriv}
\end{equation}
where
$$H_{A,B} := \frac{G\cE+E\cG-2F\cF}{2D},\quad \tilde Q_A = \frac{Q_A}{\sqrt{D}},\quad \tilde Q_B = \frac{Q_B}{\sqrt{D}}$$
and we have dropped the explicit dependence on $(i,j)$ since it is clear from context. Note that $H_{A,B}$ is in fact the mean curvature at $\rr(0,0)$. Truncating the above formula the same way as \eqref{eq:epstein_anal_s1_trunc} gives an efficient algorithm for evaluating $Z_{A,B}(1)$. 

Note that the gamma function values in \eqref{eq:epstein_anal_s1_trunc} can be reused in \eqref{eq:epstein_anal_s1_deriv}. Such reuse of gamma function values also applies to the Helmholtz kernels, reducing the cost of pre-computing correction weights by a factor of 2 when solving the Helmholtz problem using a combined field formulation (see Example 2 of Section \ref{sc:results}).

In Appendix \ref{sc:app:epstein_higher_deriv} we also provide formulae and algorithms for computing the higher parametric derivatives of $Z_A(s)$ efficiently, which are useful for constructing the higher-order corrected trapezoidal rules to be introduced next.

\begin{remark}[Helmholtz DLP]\label{rmk:helmholtz_dlp}
The correction weight of the $O(h^3)$ corrected trapezoidal rule for the Helmholtz DLP is exactly the same as the Laplace DLP. To see this, we use the expansion
\begin{equation}
(1-i\kappa r)e^{i\kappa r}\frac{\rr\cdot\nn}{r^3}  = \left(\frac{1}{r}+\frac{\kappa^2}{2}r+O(r^2)\right)\frac{\rr\cdot\nn}{r^2}
\end{equation}
where the fraction $\frac{\rr\cdot\nn}{r^2}$ is bounded near $0$. Unlike the Helmholtz SLP, there is no $O(1)$ term in this expansion, thus an $O(h^2)$ correction is \emph{not} needed.
\end{remark}

\section{High-order singular quadratures}
\label{sc:high-order_quad}

We first define some simplified notations. We will use $\uu\equiv(u,v)\in\RR^2$ and $\d\uu\equiv\d u\d v$ interchangeably, and write $$\int_{|\uu|\leq a}f(\uu)\d\uu:=\int_{-a}^a\int_{-a}^af(u,v)\d u\d v.$$
We use $O(\uu^k) = O(u^k,u^{k-1}v,\dots,v^k)$ to denote terms of order $k$ or higher and $\Theta(\uu^k)$ to denote homogeneous polynomials of order exactly $k$. We will also denote $\ii\equiv(i,j)\in\ZZ^2$ and $|\ii|:=\max(i,j)$, then the punctured summation is written as
$$
\sideset{}'\sum_{|\ii|\leq N}f(\ii) := \sum_{i=-N}^{N}\sideset{}'\sum_{j=-N}^{N}f(i,j).
$$
In addition, we will drop the explicit dependence of $Z_A(s)$ and $Q_A(u,v)$ on $A$ and write $Z(s)$ and $Q(u,v)\equiv Q(\uu)$ whenever $A$ is clear from the context.

\subsection{High-order errors as Wigner-type limits}
\label{sc:high-order_analysis}
To construct a high-order corrected trapezoidal rule for the Laplace SLP, a complete picture of the singular components in the kernel is needed. To this end, we expand the kernel $1/r(\uu)$ in terms of the fundamental form $Q(\uu)=Eu^2+2Fuv+Gv^2$ as follows. Using the binomial series
\begin{equation}
(1+x)^{-\frac{1}{2}} = 1-\frac{1}{2}x+\frac{3}{8}x^2+\dots+\binom{-\frac{1}{2}}{m}x^m+\dots,
\end{equation}
we can write
\begin{equation}
\frac{1}{r} = \frac{1}{\sqrt{Q+(r^2-Q)}} = \frac{1}{Q^{\frac{1}{2}}}-\frac{r^2-Q}{2Q^\frac{3}{2}}+\frac{3(r^2-Q)^2}{8Q^{\frac{5}{2}}}-\frac{5(r^2-Q)^3}{16Q^{\frac{7}{2}}}+\dots.
\label{eq:rinv_expan_in_1st_fund_form}
\end{equation}
Assume that $P(\uu)$ is a smooth function (e.g., $P(\uu):=\sigma(\uu)J(\uu)$), and define 
\begin{equation}
    \tilde{P}(\uu) = P(\uu)\eta(\uu)
\end{equation}
where $\eta\in C_c([-a,a]^2)$ is a smooth and compactly supported function which also satisfies
\begin{equation}
    \eta(\mathbf{0})=1,\quad \eta(\uu) \equiv \eta(-\uu).
    \label{eq:bump_func}
\end{equation}
Then the SLP can be transformed as follows
\begin{equation}
\begin{aligned}
\int_{|\uu|\leq a}\frac{\tilde{P}(\uu)}{r(\uu)}\,\d\uu &= \int_{|\uu|\leq a}\Bigg(\frac{1}{r}-\sum_{m=0}^M\binom{-\frac{1}{2}}{m}\frac{(r(\uu)^2-Q(\uu))^m}{Q(\uu)^{m+\frac{1}{2}}}\Bigg)\tilde{P}(\uu)\,\d\uu\\
&\qquad + \sum_{m=0}^M\binom{-\frac{1}{2}}{m}\int_{|\uu|\leq a}\frac{(r(\uu)^2-Q(\uu))^m}{Q(\uu)^{m+\frac{1}{2}}}\tilde{P}(\uu)\,\d\uu \\
& = \sideset{}'\sum_{|\ii|\leq N}\frac{\tilde{P}(\ii h)}{r(\ii h)}h^2 + O(h^{M+2}) \\
  &\qquad + \sum_{m=0}^M\binom{-\frac{1}{2}}{m} \Bigg(\int\limits_{|\uu|\leq a}\frac{(r(\uu)^2-Q(\uu))^m}{Q(\uu)^{m+\frac{1}{2}}}\tilde{P}(\uu)\,\d\uu \\
 & \qquad -\sideset{}'\sum_{|\ii|\leq N}\frac{(r(\ii h)^2-Q(\ii h))^m}{Q(\ii h)^{m+\frac{1}{2}}}\tilde{P}(\ii h)\,h^2\Bigg)
\end{aligned}
\label{eq:lap_corr_high_order}
\end{equation}
where in the first equality, we add and subtract the first $M+1$ terms in the expansion \eqref{eq:rinv_expan_in_1st_fund_form}; in the second equality, the punctured trapezoidal rule with $h=a/(N+\frac{1}{2})$ is applied to the first integral whose integrand is a $C^{M}_c([-a,a]^2)$ function therefore resulting in a $O(h^{M+2})$ error, then we group the integrals and summations into pairs, such that each pair corresponds to one term in the expansion \eqref{eq:rinv_expan_in_1st_fund_form}. 

We call each integral-summation pair on the right-hand side of \eqref{eq:lap_corr_high_order} a \emph{Wigner pair} and denote the $m$-th pair by $W_h^m[\tilde{P}]$, $m=0,1,2,\dots$, i.e.
\begin{equation}
    \begin{aligned}
    W_h^m[\tilde{P}] := \binom{-\frac{1}{2}}{m}&\Bigg(\int\limits_{|\uu|\leq a}\frac{\big(r^2(\uu)-Q(\uu)\big)^{m}\tilde{P}(\uu)}{\sqrt{Q(\uu)}^{2m+1}}\,\d\uu \\ 
    & \qquad - \sideset{}'\sum_{|\ii |\leq N} \frac{\big(r^2(\ii h)-Q(\ii h)\big)^{m}\tilde{P}(\ii h)}{\sqrt{Q(\ii h)}^{2m+1}}\,h^2\Bigg)
    \end{aligned}
    \label{eq:wigner_pair}
\end{equation}

To obtain practical formulae for $W_h^m[P]$, we further expand the smooth function $r^2-Q$ as a Taylor-Maclaurin series
\begin{equation}
r^2-Q = \Big(Q+q_3+q_4+q_5+\dots\Big)-Q = q_3+q_4+O(\uu^5)
\label{eq:r2-Q}
\end{equation}
where $q_k$ denotes the $\Theta(\uu^k)$ terms in the expansion such that
$$
q_2 \equiv Q = \dd^1\cdot\dd^1, \; q_3 = 2\dd^1\cdot\dd^2,\; q_4 = 2\dd^1\cdot\dd^3+\dd^2\cdot\dd^2, \; q_5 = 2\dd^1\cdot\dd^4+2\dd^2\cdot\dd^3,\; \dots
$$
where $\dd^k$ is the $\Theta(\uu^k)$ term in the Taylor series of $\rr(\uu)$ given by
\begin{align*}
\dd^1:=&\rr_u(0)u+\rr_v(0)v\\
\dd^2:=&\frac{1}{2!}\Big(\rr_{uu}(0)u^2+2\rr_{uv}(0)uv+\rr_{vv}(0)v^2\Big)\\
\dd^3:=&\frac{1}{3!}\Big(\rr_{uuu}(0)u^3+3\rr_{uuv}(0)u^2v+3\rr_{uvv}(0)uv^2+\rr_{vvv}(0)v^3\Big)
\end{align*}
and so on. In addition, let the expansions of the smooth functions $(r^2-Q)^m$ and $P$ be
\begin{align}
(r(\uu)^2-Q(\uu))^m &= \sum_{k=3m}^\infty q_k^{(m)}(\uu) = q_{3m}^{(m)}+q_{3m+1}^{(m)}+\dots \label{eq:(r2-Q)^m}\\
P(\uu) &=\sum_{k=0}^\infty p_k(\uu) = p_0+p_1+p_2 + \dots \label{eq:P_series}
\end{align}
such that $q_k^{(m)}$ and $p_k$ are the $\Theta(\uu^k)$ terms. In particular, $q_k^{(1)} \equiv q_k$ and $q_k^{(0)}=\delta_{0k}$.

\begin{lemma}
\label{lem:wigner-pair-expan}
The Wigner pair $W_h^m[\tilde{P}]$ satisfies
\begin{equation}
    W_h^m[\tilde{P}] = \sum_{k=\left\lceil\frac{m}{2}+1\right\rceil}^K \tilde{C}^{(m)}_{2k-1}(h)\,h^{2k-1} + O(h^{2K+1})
    \label{lem:wigner-pair-expan;eq:wigner-type-series}
\end{equation}
for any integer $K\geq\left\lceil\frac{m}{2}+1\right\rceil$, where the coefficients $\tilde{C}^{(m)}_{2k-1}(h)$ are given by the following \emph{Wigner-type limits}
\begin{equation}
    \begin{aligned}
    \tilde{C}^{(m)}_{2k-1}(h) = \binom{-\frac{1}{2}}{m}\sum_{\substack{r+s=2(m+k-1)\\r\geq3m,\,s\geq0}}&\Bigg(\int\limits_{|\uu|<\infty}\frac{ q_r^{(m)}(\uu)p_s(\uu)}{\sqrt{Q(\uu)}^{2m+1}}\eta(\uu h)\,\d\uu \\
    & \quad - \sideset{}'\sum_{|\ii |<\infty} \frac{ q_r^{(m)}(\ii )p_s(\ii )}{\sqrt{Q(\ii )}^{2m+1}}\eta(\ii h)\Bigg).
    \end{aligned}
\label{lem:wigner-pair-expan;eq:C}
\end{equation}
Consequently, $W_h^m[\tilde{P}]$ contains a $\Theta(h^{2k-1})$ component if and only if $0\leq m\leq2k-2$.
\end{lemma}

\begin{proof}
First define $Y:=(r^2-Q)^m\cdot P$, then using the series \eqref{eq:(r2-Q)^m} and \eqref{eq:P_series} gives
$$
\begin{aligned}
Y = (q^{(m)}_{3m}+q^{(m)}_{3m+1}+\cdots)(p_0+p_1+p_2+\dots) = \sum_{k=3m}^{2(K+m)-1}y_k + R
\end{aligned}
$$
where the remainder $R(\uu) = O(\uu^{2(K+m)})$ and
\begin{equation}
y_k = \sum_{\substack{r+s=k\\r\geq3m,\,s\geq0}}q^{(m)}_{r}p_s\label{eq:tk;lemma}
\end{equation}
represents the $\Theta(\uu^k)$ terms in the expansion of $Y$. The truncation at $k = 2(K+m)$ will prove convenient. 
Now, for each $y_k$,
$$
\begin{aligned}
&\int_{|\uu|\leq a}\frac{y_k(\uu)\eta(\uu)}{Q(\uu)^{m+\frac{1}{2}}}\,\d\uu - \sideset{}'\sum_{|\ii|\leq N} \frac{y_k(\ii h)\eta(\ii h)}{Q(\ii h)^{m+\frac{1}{2}}}\,h^2 \\
&=\Bigg(\int_{|\uu|\leq N+\frac{1}{2}}\frac{y_k(\uu)\eta(\uu h)}{Q(\uu)^{m+\frac{1}{2}}}\,\d\uu - \sideset{}'\sum_{|\ii |\leq N} \frac{y_k(\ii )\eta(\ii h)}{Q(\ii )^{m+\frac{1}{2}}}\Bigg)h^{k-2m+1}\\
&=\Bigg(\int_{|\uu|<\infty}\frac{y_k(\uu)\eta(\uu h)}{Q(\uu)^{m+\frac{1}{2}}}\,\d\uu - \sideset{}'\sum_{|\ii |<\infty} \frac{y_k(\ii )\eta(\ii h)}{Q(\ii )^{m+\frac{1}{2}}}\Bigg)h^{k-2m+1}
\end{aligned}
$$
where the last equality holds since $\eta$ is compactly supported. Note that the above expression is $O(h^{k-2m+1})$ when $k$ is even, and vanishes when $k$ is odd due to the odd integral/summation. Similarly, substituting $y_k$ with the remainder $R$ in the above derivation gives an expression that is $O(h^{2K+1})$.

Using these facts we have
\begin{equation}
\begin{aligned}
    W_h^m[\tilde{P}]
    =& \binom{-\frac{1}{2}}{m}\Bigg(\int_{|\uu|\leq a}\frac{Y(\uu)\eta(\uu)}{Q(\uu)^{m+\frac{1}{2}}}\,\d\uu - \sideset{}'\sum_{|\ii |\leq N} \frac{Y(\ii h)\eta(\ii h)}{Q(\ii h)^{m+\frac{1}{2}}}\,h^2\Bigg)\\
    =& \binom{-\frac{1}{2}}{m} \sum_{k=\left\lceil\frac{m}{2}+1\right\rceil}^{K}\Bigg(\int\limits_{|\uu|<\infty}\frac{y_{2(k+m-1)}(\uu)\eta(\uu h)}{Q(\uu)^{m+\frac{1}{2}}}\,\d\uu \\
    & \qquad\qquad\qquad\qquad - \sideset{}'\sum_{|\ii |<\infty} \frac{y_{2(k+m-1)}(\ii )\eta(\ii h)}{Q(\ii )^{m+\frac{1}{2}}}\Bigg)h^{2k-1}+O(h^{2K+1})
\end{aligned}
\end{equation}
Then substituting the above formula using \eqref{eq:tk;lemma} yields \eqref{lem:wigner-pair-expan;eq:wigner-type-series}.
\qed\end{proof}

The next theorem gives an error expansion of the punctured trapezoidal rule.

\begin{theorem}
\label{thm:high-order}
With $\eta$ as defined in \eqref{eq:bump_func}, the punctured trapezoidal rule discretization of the single-layer potential on a parametric surface centered at $\mathbf{0}$ has an error expansion of the form
\begin{equation}
    \int_{|\uu|\leq a}\frac{\tilde{P}(\uu)}{r(\uu)}\,\d\uu =\sideset{}'\sum_{|\ii |\leq N}\frac{\tilde{P}(\ii h)}{r(\ii h)}h^2 + \sum_{k=1}^K \tilde{C}_{2k-1}(h)\,h^{2k-1} + O(h^{2K+1})
    \label{thm:high-order;eq:error_expan}
\end{equation}
for any integer $K>0$. The coefficients $\tilde{C}_{2k-1}$ for any $k\geq1$ are given by
\begin{equation}
\tilde{C}_{2k-1}(h) := \sum_{m=0}^{2k-2} \tilde{C}^{(m)}_{2k-1}(h)
\label{thm:high-order;eq:err_coeff}
\end{equation}
where for each $m$, the constants $\tilde{C}^{(m)}_{2k-1}(h)$ are the Wigner-type limits \eqref{lem:wigner-pair-expan;eq:C} that associate with the Wigner pair $W_h^m[\tilde{P}]$.
\end{theorem}
\begin{proof}
Equation \eqref{eq:lap_corr_high_order} can be written as
\begin{equation}
    \int_{|\uu|\leq a}\frac{\tilde{P}(\uu)}{r(\uu)}\,\d\uu =\sideset{}'\sum_{|\ii |\leq N}\frac{\tilde{P}(\ii h)}{r(\ii h)}h^2 + \sum_{m=0}^M W_h^m[\tilde{P}] + O(h^{M+2})
    \label{thm:high-order;eq:lap_corr_high_order}
\end{equation}
Then letting $M=2K-2$, substituting the Wigner pairs $W_h^m[\tilde{P}]$ by the estimate \eqref{lem:wigner-pair-expan;eq:wigner-type-series} and interchanging summation order yield \eqref{thm:high-order;eq:error_expan}; note that the error is $O(h^{2K+1})$ instead of $O(h^{2K})$ since the even powers of $h$ vanish due to the odd integrals/summations, as stated in the proof of Lemma \ref{lem:wigner-pair-expan}.
\qed\end{proof}

\subsection{Converged Wigner-type limits and derivatives of Epstein zeta}
\label{sc:high-order:converged_Wigner-type}

The Wigner-type limits $\tilde{C}^{(m)}_{2k-1}(h)$ as defined by \eqref{lem:wigner-pair-expan;eq:C} cannot be computed efficiently. However, using the similar ideas from Sections \ref{sc:slp_wigner} and \ref{sc:dlp_epstein_der}, we will show that the converged values $\lim_{h\to0}\tilde{C}^{(m)}_{2k-1}(h)$ can be efficiently evaluated as Epstein zeta function values or its parametric derivatives. Furthermore, we will show in Theorem \ref{thm:high-order-with-converged-wigner} that the error analysis in Theorem \ref{thm:high-order} remains valid when these zeta function approximations are used.

To further simplify notations, we denote
\begin{equation}
    \cL^m_h[f]:= \lim_{N\to\infty}\sideset{}'\sum_{|\ii |<N} \frac{ f(\ii )}{Q(\ii )^{m+\frac{1}{2}}}\eta(\ii h) - \int_{|\uu|<N+\frac{1}{2}}\frac{ f(\uu)}{Q(\uu)^{m+\frac{1}{2}}}\eta(\uu h)\,\d\uu
    \label{eq:unconverged_wigner-type}
\end{equation}
then \eqref{lem:wigner-pair-expan;eq:C} can be written as
\begin{equation}
    \tilde{C}^{(m)}_{2k-1}(h) = \binom{-\frac{1}{2}}{m}\sum_{\substack{r+s=2(m+k-1)\\r\geq3m,\,s\geq0}}-\cL^m_h\left[q^{(m)}_rp_s\right],\quad k\geq\left\lceil\frac{m}{2}+1\right\rceil.
    \label{eq:unconverged_wigner-type_err}
\end{equation}
Note that each term $q^{(m)}_rp_s$ is a homogeneous polynomial consisting of $\Theta(u^{2d-l}v^l)$ terms for some even integer $2d$ and for $l=0,1,\dots,2d$. Therefore $\tilde{C}^{(m)}_{2k-1}(h)$ is a linear combination of terms of the form $\cL^m_h\left[u^{2d-l}v^l\right]$, $\frac{3m}{2}\leq d\leq m+K-1$.

\begin{theorem}
    \label{thm:converged_wigner-type}
    Define the \emph{converged Wigner-type limits} as
    \begin{equation}
        \cL^m\left[u^{2d-l}v^l\right]:=\lim_{h\to0} \cL^m_h\left[u^{2d-l}v^l\right],\quad 0\leq l\leq d\in \NN, d\geq \frac{3m}{2}
    \end{equation}
    then
    \begin{equation}
    \cL^m\left[u^{2d-l}v^l\right]= 
     \begin{cases}
        \displaystyle \frac{1}{d! \binom{d-m-\frac{1}{2}}{d}}\left(\frac{\pd}{\pd E}\right)^{d-l} \left(\frac{1}{2}\frac{\pd}{\pd F}\right)^lZ\big(2(m-d)+1\big) & l\leq d\\[15pt]
        \displaystyle \frac{1}{d! \binom{d-m-\frac{1}{2}}{d}} \left(\frac{1}{2}\frac{\pd}{\pd F}\right)^{2d-l}\left(\frac{\pd}{\pd G}\right)^{l-d}Z\big(2(m-d)+1\big) & l>d
      \end{cases}
    \label{eq:wigner-type_epstein}
    \end{equation}
\end{theorem}
\begin{proof}
Using Theorem \ref{thm:wigner_lim} and the fact that $\lim_{\uu\to\mathbf{0}}\eta(\uu)=1$, the following holds for all $0<\mathrm{Re}\,s<2$,
\begin{equation}
\begin{aligned}
\lim_{h\to0} \cL^{(s-1)/2}_h\left[1\right] &= \lim_{h\to0} \Bigg(\sideset{}'\sum_{|\ii |<\infty} \frac{ 1}{Q(\ii )^{\frac{s}{2}}}\eta(\ii h) - \int\limits_{|\uu|<\infty}\frac{ 1}{Q(\uu)^{\frac{s}{2}}}\eta(\uu h)\,\d\uu\Bigg)\\
& = \sideset{}'\sum_{|\ii |<\infty} \frac{ 1}{Q(\ii )^{\frac{s}{2}}} - \int\limits_{|\uu|<\infty}\frac{ 1}{Q(\uu)^{\frac{s}{2}}}\,\d\uu = Z(s)
\end{aligned}
\label{eq:limiting;thm:converged_wigner-type}
\end{equation}
Furthermore, note that both $\cL^{(s-1)/2}_h\left[1\right]$ and the zeta function $Z(s)$ are analytic functions of $s$ and $Z(s)$ is defined for all $s\in\CC$ (except at $s=2$), therefore
\begin{equation}
    Z(s) \equiv \lim_{h\to0} \cL^{(s-1)/2}_h\left[1\right]
    \label{eq:Z;thm:converged_wigner-type}
\end{equation}
holds for all $s\in\CC\setminus{\{2\}}$. In particular setting $s=1$ in \eqref{eq:limiting;thm:converged_wigner-type} proves \eqref{eq:wigner-type_epstein} for the case $m=d=0$. 

Next, consider $\cL^m_h\left[u^{2d-l}v^l\right]$ for $d>1$ and assume $l\leq d$. By \eqref{eq:Z;thm:converged_wigner-type}, we have
$$Z(s-2d) = \lim_{h\to0} \cL^{(s-2d-1)/2}_h\left[1\right],$$
then apply $\left(\frac{\pd}{\pd E}\right)^{d-l} \left(\frac{1}{2}\frac{\pd}{\pd F}\right)^l$ to both sides and take the parametric derivative under the limit sign yields
\begin{equation}
\begin{aligned}
    \left(\frac{\pd}{\pd E}\right)^{d-l} \left(\frac{1}{2}\frac{\pd}{\pd F}\right)^lZ(s-2d) &= \lim_{h\to0} \left(\frac{\pd}{\pd E}\right)^{d-l} \left(\frac{1}{2}\frac{\pd}{\pd F}\right)^l\,\cL^{(s-2d-1)/2}_h\left[1\right]\\
    &= \lim_{h\to0} d!\,\binom{d-\frac{s}{2}}{d}\,\cL^{(s-1)/2}_h\left[u^{2d-l}v^l\right]\\
    &=  d!\,\binom{d-\frac{s}{2}}{d}\, \cL^{(s-1)/2}\left[u^{2d-l}v^l\right]
\end{aligned}
\end{equation}
which by setting $s = 2m+1$ yields \eqref{eq:wigner-type_epstein}. The case where $l>d$ is proved similarly.
\qed
\end{proof}

Theorem \ref{thm:converged_wigner-type} is a generalization of Theorem \ref{thm:wigner_lim} to the parametric derivatives of the Epstein zeta function, which is the key for the efficient construction of high-order correction weights for the punctured trapezoidal rule. It allows us to approximate the error coefficients $\tilde{C}^{(m)}_{2k-1}(h)$ by their converged values:
\begin{equation}
    C^{(m)}_{2k-1} := \lim_{h\to0}\tilde{C}^{(m)}_{2k-1}(h) = \binom{-\frac{1}{2}}{m}\sum_{\substack{r+s=2(m+k-1)\\r\geq3m,\,s\geq0}}-\cL^m\left[q^{(m)}_rp_s\right]
    \label{eq:converged_wigner-type}
\end{equation}
which in turn can be efficiently evaluated as parametric derivatives of the Epstein zeta function $Z(s)$. Detailed formulae for such computation are included in Appendix \ref{sc:app:epstein_higher_deriv}.

Finally, we are ready to state the key analytical result of this paper:
\begin{theorem}
\label{thm:high-order-with-converged-wigner}
Let $\eta(\uu)\in C^{4K}_c([-a,a]^2)$ be as defined in \eqref{eq:bump_func} such that it is $4K$ times continuously differentiable. In addition, assume that
\begin{equation}
    \frac{\partial^{\alpha+\beta}}{\partial u^\alpha\partial v^\beta}\eta(\mathbf{0}) = 0,
\end{equation}
for all non-negative integers $\alpha$ and $\beta$ such that
\begin{equation}
    0<\alpha+\beta\leq 2K,
\end{equation}
i.e., all the derivatives of order up to $2K$ vanish at $\mathbf{0}$. Then replacing $\tilde{C}_{2k-1}(h)$ in \eqref{thm:high-order;eq:error_expan} by its converged value (using the definitions \eqref{thm:high-order;eq:err_coeff} and \eqref{eq:converged_wigner-type})
\begin{equation}
    C_{2k-1}:=\lim_{h\to0}\tilde{C}_{2k-1}(h)=\sum_{m=0}^{2k-2} C^{(m)}_{2k-1}
    \label{eq:converged_err_coeff}
\end{equation}
does not change the error estimate, that is,
\begin{equation}
    \int_{|\uu|\leq a}\frac{\tilde{P}(\uu)}{r(\uu)}\,\d\uu =\sideset{}'\sum_{|\ii |\leq N}\frac{\tilde{P}(\ii h)}{r(\ii h)}h^2 + \sum_{k=1}^K C_{2k-1}\,h^{2k-1} + O(h^{2K+1})
    \label{eq:error_expan_converged}
\end{equation}
\end{theorem}
\begin{proof}
Applying Theorem \ref{thm:wigner-type_converge_rate} in Appendix \ref{sc:app:proof} to each term of \eqref{eq:unconverged_wigner-type_err} yields
$$
\big|C^{(m)}_{2k-1} -\tilde{C}^{(m)}_{2k-1}(h)\big| = O(h^{2K})
$$
for all $0\leq m < 2k-1 \leq 2K-1$. Then using the definition \eqref{thm:high-order;eq:err_coeff} yields
\begin{align*}
    &\big|C_{2k-1}-\tilde{C}_{2k-1}(h)\big| \leq \sum_{m=0}^{2k-2} \big|C^{(m)}_{2k-1} -\tilde{C}^{(m)}_{2k-1}(h)\big| = O(h^{2K})
\end{align*}
for all $1\leq k \leq K$, which implies that
\begin{align*}
    \Bigg|\int_{|\uu|\leq a}&\frac{\tilde{P}(\uu)}{r(\uu)}\,\d\uu -\sideset{}'\sum_{|\ii |\leq N}\frac{\tilde{P}(\ii h)}{r(\ii h)}h^2 - \sum_{k=1}^K C_{2k-1}\,h^{2k-1} \Bigg| \\
    \leq &\Bigg|\int_{|\uu|\leq a}\frac{\tilde{P}(\uu)}{r(\uu)}\,\d\uu -\sideset{}'\sum_{|\ii |\leq N}\frac{\tilde{P}(\ii h)}{r(\ii h)}h^2 - \sum_{k=1}^K \tilde{C}_{2k-1}(h)\,h^{2k-1}\Bigg| \\
    &+\sum_{k=1}^K \big|C_{2k-1}-\tilde{C}_{2k-1}(h)\big|\,h^{2k-1}\\
    \leq & O(h^{2K+1}) + h\,\sum_{k=1}^K \big|C_{2k-1}-\tilde{C}_{2k-1}(h)\big|\\
    =& O(h^{2K+1})
\end{align*}
where the second inequality used the estimate \eqref{thm:high-order;eq:error_expan}.
\qed\end{proof}

To summarize the analysis in this and the previous sections, the punctured trapezoidal rule has an error expansion \eqref{eq:error_expan_converged} where the coefficients $C_{2k-1}$ are given by (combining \eqref{eq:converged_wigner-type} and \eqref{eq:converged_err_coeff}):
\begin{equation}
\boxed{
    C_{2k-1} = \sum_{m=0}^{2k-2} C^{(m)}_{2k-1} = \sum_{m=0}^{2k-2}-\binom{-\frac{1}{2}}{m}\cL^m\Bigg[\sum_{\substack{r+s=2(m+k-1)\\r\geq3m,\,s\geq0}}q^{(m)}_rp_s\Bigg]}
    \label{eq:converged_wigner-type_coeff}
\end{equation}
where $q^{(m)}_r$ and $p_s$ are terms in the expansions (\ref{eq:(r2-Q)^m}--\ref{eq:P_series}). The converged Wigner-type limits $\cL^m[\cdot]$ can be efficiently computed using the zeta function connection \eqref{eq:wigner-type_epstein} (Theorem \ref{thm:converged_wigner-type}), and the use of these limits in the error analysis for intermediate $h$ values is justified by Theorem \ref{thm:high-order-with-converged-wigner}.

\paragraph{Error components associated with the $O(h^5)$ quadrature.}

As an application, we show how the error components associated with the $O(h^5)$ corrected quadrature can be computed.

According to Theorem \ref{thm:high-order}, a $O(h^5)$ trapezoidal rule for the SLP requires correcting the $\Theta(h)$ and $\Theta(h^3)$ errors underlying the first three Wigner pairs. By Theorem \ref{thm:high-order-with-converged-wigner}, the error coefficients can be approximated by the converged Wigner-type limits  \eqref{eq:converged_wigner-type_coeff}, given by
$$
\begin{aligned}
C_1 &= C_1^{(0)} = -\cL^0[q_0^{(0)}p_0] \\
C_3 &= C_3^{(0)}+C_3^{(1)}+C_3^{(2)} = -\cL^0[q_0^{(0)}p_2]-\left(-\frac{1}{2}\right)\cL^1[q_4^{(1)}p_0+q_3^{(1)}p_1]-\frac{3}{8}\cL^2[q_6^{(2)}p_0].
\end{aligned}
$$
Based on the definition \eqref{eq:unconverged_wigner-type} of $\cL^m_h[\cdot]$, we define the notation
\begin{equation}
    \cL[f]:= \lim_{h\to0} \lim_{N\to\infty}\left(\sideset{}'\sum_{|\ii |<N}  f(\ii )\,\eta(\ii h) - \int_{|\uu|<N+\frac{1}{2}} f(\uu)\,\eta(\uu h)\,\d\uu\right),
    \label{def:wigner-type_notation}
\end{equation}
then $C_1$ and $C_3$ can be written as converged Wigner-type limits:
\begin{equation}
    \begin{aligned}
    C_1 &= -\cL[\cW^{(1)}(\uu)p_0(\uu)]\\
    C_3 &= -\cL[\cW^{(1)}(\uu)p_2(\uu)]-\cL[\cW^{(2)}(\uu)p_1(\uu)]-\cL[\cW^{(3)}(\uu)p_0(\uu)]
    \end{aligned}
    \label{eq:converged_wigner-type_err_O(h^5)}
\end{equation}
where we have re-grouped the terms in $C_3^{(0)}, C_3^{(1)}$ and $C_3^{(2)}$ by $p_k=\Theta(\uu^k)$, and where the functions $\cW^{(m)}(\uu)$ are defined as
\begin{align}
\cW^{(1)}(\uu) &:=\frac{q_0^{(0)}}{\sqrt{Q}}=\frac{1}{\sqrt{Q}},
\label{eq:w1_wigner-type}\\
\cW^{(2)}(\uu) &:= -\frac{1}{2}\frac{q_3^{(1)}}{Q^{\frac{3}{2}}},
\label{eq:w2_wigner-type}\\
\cW^{(3)}(\uu) &:= -\frac{1}{2}\frac{q_4^{(1)}}{Q^{\frac{3}{2}}}+\frac{3}{8}\frac{q_6^{(2)}}{Q^{\frac{5}{2}}},
\label{eq:w3_wigner-type}
\end{align}
which only depend on the SLP kernel and the geometry (also see Table \ref{tbl:slp_singularity_summary} for a systematic picture). To summarize, we define the following Wigner-type limits that will be useful for computing the correction weights in the next section.
\begin{equation}
\begin{aligned}
D_0 &= -\cL[\cW^{(1)}(\uu)]\,h - \cL[\cW^{(3)}(\uu)]\,h^3 = \cC_0\,h+\cC_{01}h^3=-Z(1)\,h+\cC_{01}h^3,\\
D_1 &= -\cL[\cW^{(2)}(\uu)\,u]\,h^2 = \cC_1h^2,\\
D_2 &= -\cL[\cW^{(2)}(\uu)\,v]\,h^2 = \cC_2h^2,\\
D_3 &= -\cL[\cW^{(1)}(\uu)\,u^2]\,h = \cC_3\,h\equiv -2\frac{\partial}{\partial E} Z(-1)\,h,\\
D_4 &= -\cL[\cW^{(1)}(\uu)\,v^2]\,h = \cC_4\,h\equiv -2\frac{\partial}{\partial G} Z(-1)\,h,\\
D_5 &= -\cL[\cW^{(1)}(\uu)\,u\,v]\,h = \cC_5\,h\equiv -\frac{\partial}{\partial F} Z(-1)h.
\end{aligned}
\label{eq:D0-D5_lap_slp}
\end{equation}
A number of remarks are in order. Firstly, these converged Wigner-type limits correspond to \eqref{eq:converged_wigner-type_err_O(h^5)}, where $p_0,p_1,p_2,\dots$ are replaced by the corresponding monomials so that they serve as a basis for smooth functions. 
Secondly, each Wigner-type limit associated with $\cW^{(m)}$ is multiplied by $h^{m}$, which corresponds to the right-hand side of the moment equations \eqref{eq:wei_moment_eqs} after cancelling  $h^{\alpha+\beta}$ on both sides; this will prove convenient when we later construct the correction weights. 
Finally, the fact that the constants $\cC_0,\cC_3,\cC_4,\cC_5$ are given by $Z(s)$ and its derivatives is an application of Theorem \ref{thm:converged_wigner-type}; $\cC_1,\cC_2$ and $\cC_{01}$ have similar expressions but involve higher derivatives, thus we leave the formulae in Appendix \ref{sc:app:slp_wigner-type}. Formulae for evaluating the zeta function $Z(s)$ and its derivatives are included in Appendix \ref{sc:app:epstein_higher_deriv}.

Table \ref{tbl:slp_singularity_summary} presents a systematic approach to identify the functions that are necessary for defining $D_0$--$D_5$ of \eqref{eq:D0-D5_lap_slp}. (Table \ref{tbl:dlp_singularity_summary} in the Appendix is a similar table for the Laplace DLP.)

\begin{table}[htbp]
\centering
\begin{tabular}{l | lll | l}
$\frac{1}{r}$ & $\frac{1}{\sqrt{Q}}$ &  $-\frac{1}{2}\frac{r^2-Q}{\sqrt{Q}^3}$ & $\frac{3}{8}\frac{(r^2-Q)^2}{\sqrt{Q}^5}$ & Density\\[10pt]
\hline
$\cW^{(1)}$ & $\frac{1}{\sqrt{Q}}$ &  &  & $O(\uu^2),O(\uu^0)$\\[10pt]
$\cW^{(2)}$ & 0 & $-\frac{\,q^{(1)}_3}{2\sqrt{Q}^3}$  &  & $O(\uu^1)$\\[10pt]
$\cW^{(3)}$ & 0 & $-\frac{\,q^{(1)}_4}{2\sqrt{Q}^3}$  & $\frac{3\,q^{(2)}_6 }{8\sqrt{Q}^5}$  & $O(\uu^0)$
\end{tabular}
\caption{Summary of the singularity information in the Laplace SLP kernel. $q_k^{(m)}$ are as defined in \eqref{eq:(r2-Q)^m}. $\cW^{(i)}$ denotes the sum of the singular functions in the same row; when $\cW^{(i)}$ is multiplied by a $\Theta(\uu^j)$ component of the density function, the corresponding Wigner-type limit is $O(h^{i+j})$. For a $O(h^{2K+1})$ corrected trapezoidal rule, the Wigner-type limits associated with $\cW^{(i)}\cdot \Theta(\uu^j)$ for $i+j\leq 2K-1$, $i+j$ odd, account for all the errors that require correction. Right column of the table lists all such $O(\uu^j)$ density function components associated with $\cW^{(i)}$ for $K=2$.}\label{tbl:slp_singularity_summary}
\end{table}

In general, the problem of extracting the error components for the construction of a $O(h^{2K+1})$ corrected trapezoidal rule can be broken down into 3 steps: 
\begin{enumerate}[1)]
    \item Expand the numerators of the first $2K-1$ Wigner pairs in \eqref{eq:lap_corr_high_order}.
    \item Based on (\ref{eq:error_expan_converged},~\ref{eq:converged_wigner-type_coeff}), extract the functions $\cW^{(m)}(\uu)$ involved in the $\Theta(h^{2k-1})$ errors for $k=1,\dots,K$. (Table \ref{tbl:slp_singularity_summary} [or Table \ref{tbl:dlp_singularity_summary} in Appendix \ref{sc:app:dlp_corr_h5}] presents a systematic way of doing this.)
    \item Compute the error components as converged Wigner-type limits similar to \eqref{eq:D0-D5_lap_slp}, which are efficiently evaluated as appropriate Epstein zeta function values and its parametric derivatives according to Theorem \ref{thm:converged_wigner-type}.
\end{enumerate}

\begin{remark}
\label{rmk:larger_stencil}
The fact that generally $D_5\neq0$ is in stark contrast to the flat-space trapezoidal rules in $\RR^2$, such as those in \cite{duan2009high} and \cite{marin2014corrected} where the analogous quantities always vanish. Consequently, a larger stencil will be required on a curved surface to accommodate these extra degrees-of-freedom. See Figure \ref{fig:stencil}.
\end{remark}

\begin{remark}
Helsing in \cite{helsing2013higher} designed a panel-based, singularity-subtraction quadrature method for the Laplace DLP. An expansion analogous to \eqref{eq:rinv_expan_in_1st_fund_form} (i.e., \eqref{eq:dlp_expan_in_1st_fund_form} in Appendix \ref{sc:app:dlp_corr_h5}) is used, where the smooth terms such as $(r^2-Q)^m$ are interpolated as a polynomial $P$ on each panel, then the singular functions ${P}/{Q^{m+\frac{1}{2}}}$ are integrated analytically and recursively. The numerical example in \cite{helsing2013higher} used only the first 2 terms in the singular expansion, which would only give a $O(h^3)$ scheme, but due to the dominant interpolation error when $h$ is large, initially $O(h^{10})$ convergence was observed before the kernel's $O(h^3)$ error finally dominates; in contrast, our scheme is seeing this true $O(h^3)$ error from the beginning using only one term in the singular expansion (see Figure \ref{fig:convergence_patch}).
\end{remark}

\subsection{Computing the local correction weights of high-order quadratures}
\label{sc:high-order:weights}

To construct a $O(h^{2K+1})$ quadrature via local correction, we first define a stencil that corresponds to the quadrature points in the vicinity of the singularity. Let $U_K, K\geq1,$ be the stencil for the $O(h^{2K+1})$ quadrature on a curved surface, and also define $U^\mathrm{flat}_K$ to be the corresponding stencil in the flat space $\RR^2$ (see Remark \ref{rmk:larger_stencil} above), such that
\begin{align}
    U_K &:= \{(\mu,\nu)\in\ZZ^2\,:\,|\mu|+|\nu|\leq K, \max\{|\mu|,|\nu|\} < K\} \label{eq:stencil_curved}\\
    U^\mathrm{flat}_K &:= \{(\mu,\nu)\in\ZZ^2\,:\,|\mu|+|\nu|\leq K-1\} \label{eq:stencil_flat}
\end{align}
The stencils $U^\mathrm{flat}_K$ are used in \cite{duan2009high,marin2014corrected} to construct quadratures on $\RR^2$. See Figure \ref{fig:stencil} for a comparison between $U_K$ and $U^\mathrm{flat}_K$. Thus 
\begin{align}
    |U_K|&=2K(K+1)-3\\
    |U^\mathrm{flat}_K|&=2K(K-1)+1
\end{align}
Notice the relationship:
$$U_K = \{(\mu,\nu)\in U^\mathrm{flat}_{K+1}\,:\,|\mu|,|\nu| < k\}$$
that is, $U_K$ is $U^\mathrm{flat}_{K+1}$ excluding the 4 points farthest away from $(0,0)$.

\begin{figure}[htbp]
    \centering
    \includegraphics[width=\textwidth]{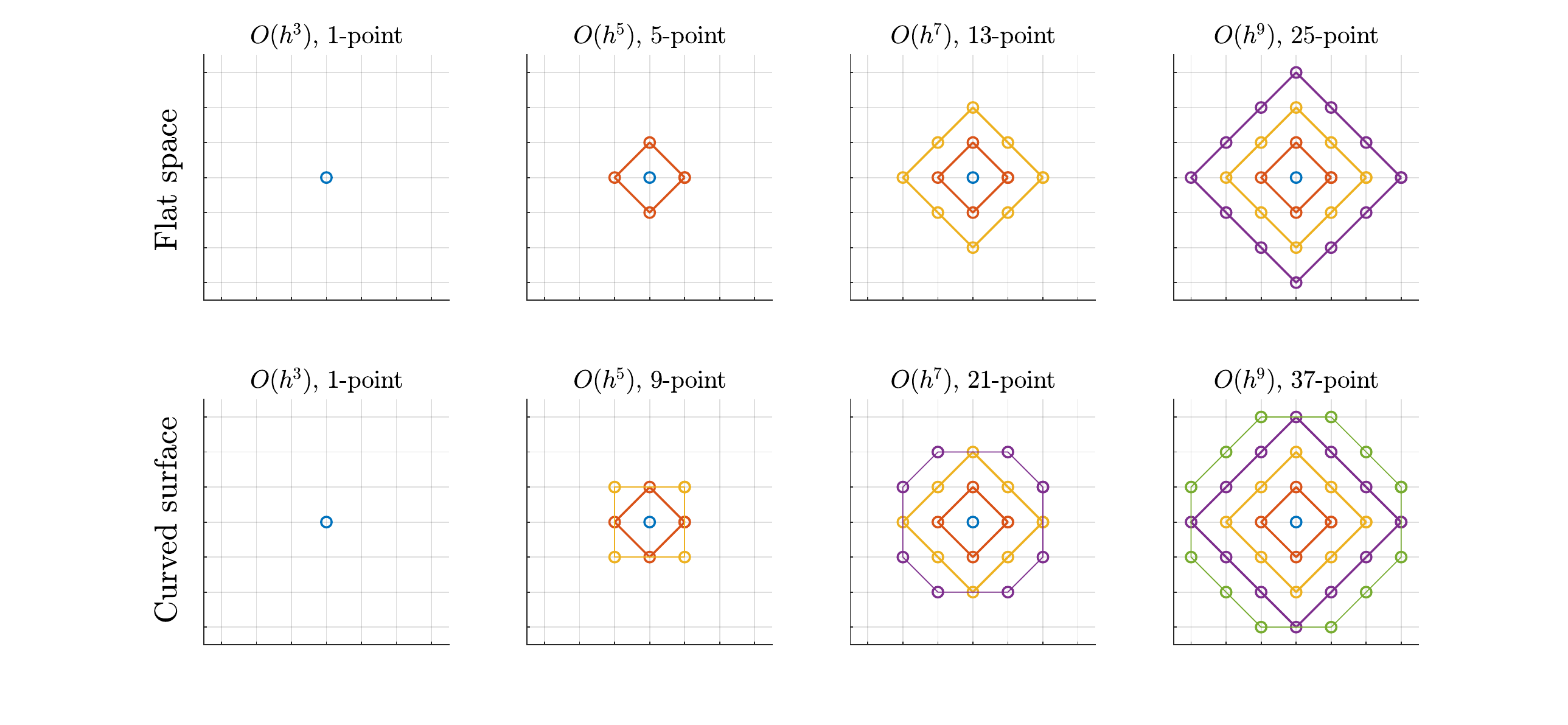}
    \caption{Comparison of $U^\mathrm{flat}_K$ (top row), the correction stencils associated with the flat space $\RR^2$, and $U_K$ (bottom row), stencils associated with the parameter space of a curved surface, $K=1,2,3,4$; see the definitions (\ref{eq:stencil_curved},~\ref{eq:stencil_flat}). The order, $O(h^{2K+1})$, of the associated corrected quadrature for the $\frac{1}{r}$ type singularity is shown for each stencil. The colors are meant for comparison of how the stencil grows. Compared to $U^\mathrm{flat}_K$, the extra points needed by $U_K$ are due to the extra degrees-of-freedom required by the quantities such as $D_5$ in \eqref{eq:D0-D5_lap_slp}. }
    \label{fig:stencil}
\end{figure}

Using the idea of moment fitting (e.g., see \cite{wissmann1986partially}), we impose the following conditions on the unknown correction weights $\tau_{\mu,\nu}$, $(\mu,\nu)\in U_k$.
\begin{enumerate}[(1)]
    \item The moment equations
        \begin{equation}
        \begin{aligned}
        \sum_{(\mu,\nu)\in U_K}(\mu h)^{\alpha}(\nu h)^{\beta}\tau_{\mu,\nu} =&  \sum_{m\in M(\alpha,\beta)}-\cL\big[\cW^{(m)}(\uu){u^{\alpha}v^{\beta}}\big]h^{m+\alpha+\beta},\\
        M(\alpha,\beta):=&\{m\in\NN\,:\,m+\alpha+\beta = 2k-1, 1\leq k\leq K\}.
        \end{aligned}
        \label{eq:wei_moment_eqs}
        \end{equation}
        for $\alpha,\beta\geq0$ and $0\leq\alpha+\beta\leq2(K-1)$. The functions $\cW^{(m)}(\uu)u^{\alpha}v^{\beta}$ are identified systematically based on the singular kernel (e.g., see Tables \ref{tbl:slp_singularity_summary} or \ref{tbl:dlp_singularity_summary}). For conceptual understanding, here we have kept the common factor $h^{\alpha+\beta}$ on both sides of \eqref{eq:wei_moment_eqs} that associates with the monomial $u^\alpha v^\beta$, in practice they should be cancelled to simply the equations, as done in \eqref{eq:corr_weights_system}. There are $K(2K-1)$ equations in \eqref{eq:wei_moment_eqs}.
    \item The symmetric conditions
        \begin{equation}
            \tau_{\mu,\nu} = \tau_{-\mu,-\nu},\quad \text{for all }\tau_{\mu,\nu}\in U_K\setminus U^\mathrm{flat}_K
            \label{eq:wei_sym_cond}
        \end{equation}
        This gives another $2(K-1)$ equations.
    \item The normalizing conditions
        \begin{equation}
            \tau_{\mu,\nu} +  \tau_{-\mu,\nu} + \tau_{\mu,-\nu}+ \tau_{-\mu,-\nu} = 0,\quad \text{for all }\tau_{\mu,\nu}\in U_K\setminus U^\mathrm{flat}_K
        \label{eq:wei_normalize_cond}
        \end{equation}
         This gives another $(K-1)$ equations.
\end{enumerate}
Together, (\ref{eq:wei_moment_eqs}--\ref{eq:wei_normalize_cond}) have $2K(K+1)-3$ equations which match the number of unknowns $|U_K|$. 

For an arbitrary $K>0$, we do not have a proof that the square system (\ref{eq:wei_moment_eqs}--\ref{eq:wei_normalize_cond}) is non-singular. However, for a given $K$ in practice, this can be easily verified using computer algebra packages such as Mathematica. In particular, we have:

\begin{theorem}
For $1\leq K\leq5$, a $O(h^{2K+1})$ corrected trapezoidal rule for the Laplace SLP is given by
\begin{equation}
\int_{|\uu|\leq a}\frac{P(\uu)}{r(\uu)}\d\uu = \sideset{}'\sum_{|\ii |\leq N}\frac{P(\ii h)}{r(\ii h)} + \sum_{(\mu,\nu)\in U_K}P(\mu h,\nu h)\tau_{\mu,\nu} +O(h^{2K+1})
\label{eq:corr_trap_rule_slp}
\end{equation}
where the correction weights $\tau_{\mu,\nu}$ must satisfy the conditions (\ref{eq:wei_moment_eqs}--\ref{eq:wei_normalize_cond}).
\end{theorem}
\begin{proof}
For $K=1,\dots,5$, it can be verified that the matrix determinants associated with the system (\ref{eq:wei_moment_eqs}--\ref{eq:wei_normalize_cond}) are given by
\begin{center}
\def\arraystretch{1.2}
\begin{tabular}{c|c|c|c|c|c}
& $K=1$ & $K=2$ & $K=3$ & $K=4$ & $K=5$\\
\hline
determinant & $1$ & $2^5$ & $2^{22}3^6$ & $2^{65}3^{22}5^6$ & $2^{146}3^{54}5^{20}7^6$
\end{tabular}
\end{center}
therefore the weights $\tau_{\mu,\nu}$ are well defined. The moment fitting conditions \eqref{eq:wei_moment_eqs} guarantee that the correction weights fit the error components in the expansion \eqref{eq:error_expan_converged} up to (not including) the order $O(h^{2K+1})$, which implies \eqref{eq:corr_trap_rule_slp}.
\qed\end{proof}

\paragraph{Correction weights of the $O(h^5)$ quadrature.} As an application we continue to construct the $O(h^5)$ corrected trapezoidal rule. Using the $9$-point stencil $U_2$ and the quantities $D_0$--$D_5$ defined in \eqref{eq:D0-D5_lap_slp}, the conditions (\ref{eq:wei_moment_eqs}--\ref{eq:wei_normalize_cond}) yield the linear system:
\begin{equation}
\setlength{\arraycolsep}{4pt}
\left[\begin{array}{@{}*{9}{r}@{}}
1& 1 & 1 & 1 & 1 & 1 & 1 & 1 & 1\\
0& 1 &-1 & 0 & 0 & 1 &-1 & 1 &-1\\
0& 0 & 0 & 1 &-1 & 1 &-1 &-1 & 1\\
0& 1 & 1 & 0 & 0 & 1 & 1 & 1 & 1\\
0& 0 & 0 & 1 & 1 & 1 & 1 & 1 & 1\\
0& 0 & 0 & 0 & 0 & 1 & 1 &-1 &-1\\
0&0&0&0&0&1&1&1&1\\
0&0&0&0&0&1&-1&0&0\\
0&0&0&0&0&0&0&1&-1
\end{array}\right]
\left[\begin{array}{l}
\tau_{0,0}\\ \tau_{1,0}\\ \tau_{-1,0}\\ \tau_{0,1}\\ \tau_{0,-1}\\ \tau_{1,1}\\ \tau_{-1,-1}\\ \tau_{1,-1}\\ \tau_{-1,1}
\end{array}\right]
=
\begin{bmatrix}
D_0\\ D_1\\ D_2\\ D_3\\ D_4\\ D_5\\ 0\\ 0\\ 0
\end{bmatrix}
\quad
\begin{matrix}
1\\ u\\ v\\ u^2\\ v^2\\ uv\\ \, \\ \, \\ \,
\end{matrix}
\label{eq:corr_weights_system}
\end{equation}
where the first six equations correspond to \eqref{eq:wei_moment_eqs} and are labeled with the associating monomials $u^{\alpha}v^{\beta}$ on the right; the final three equations are the symmetric and normalizing conditions (\ref{eq:wei_sym_cond}--\ref{eq:wei_normalize_cond}). Solving this system gives
\begin{equation}
\begin{aligned}
\tau_{0,0} &= D_0 - D_3-D_4, \quad \tau_{\pm1,0} = \frac{D_3\pm D_1}{2},\quad\tau_{0,\pm1} = \frac{D_4\pm D_2}{2},\\
\tau_{1,1} &= \tau_{-1,-1} = -\tau_{1,-1} = -\tau_{-1,1}=\frac{D_5}{4}.
\end{aligned}
\label{eq:9pt_weights}
\end{equation}

Similarly, $O(h^5)$ corrected trapezoidal rules are constructed for the Laplace DLP in Appendix \ref{sc:app:dlp_corr_h5}, for the normal gradient of Laplace SLP in Appendix \ref{sc:app:slpn_corr_h5}, and for the Helmholtz potentials in Appendix \ref{sc:app:helm_corr_h5}.

\section{Numerical results}
\label{sc:results}
In this section, we present results of applying the locally corrected trapezoidal rules to the Laplace and Helmholtz potentials. The $O(h^3)$ quadratures for Laplace or Helmholtz SLP and DLP are constructed as in Sections \ref{sc:slp_wigner} and \ref{sc:dlp_epstein_der}; the $O(h^5)$ quadratures for Laplace SLP, DLP, and normal gradient of the SLP are constructed as in Section \ref{sc:high-order:weights} and Appendices \ref{sc:app:dlp_corr_h5} and \ref{sc:app:slpn_corr_h5}, respectively, while the corresponding quadratures for Helmholtz potentials are given in Appendix \ref{sc:app:helm_corr_h5}; the $O(h^3)$ quadrature for the Helmholtz SLP's normal gradient is obtained in an obvious manner from the corresponding $O(h^5)$ quadrature. All algorithms are implemented in MATLAB 2019a.

\paragraph{Example 1.} Figure \ref{fig:convergence_patch} shows the results of applying the corrected trapezoidal rules to the Laplace/Helmholtz layer potentials on a randomly generated quartic surface patch. The density function is also randomly generated and has the form
\begin{equation}
    \sigma(u,v) = -\left(a\cos(a+u)+b\sin(b+v)\right)e^{-c(u^2+v^2)^4},
    \label{eq:dens_func}
\end{equation}
where $a,b$ are standard Gaussian random numbers, and where $c=640$ such that $\sigma$ is compactly supported (up to double-precision) on the patch. $O(h^3)$ and $O(h^5)$ convergence are observed as expected.

\begin{figure}[htbp]
\centering
\includegraphics[width=\textwidth]{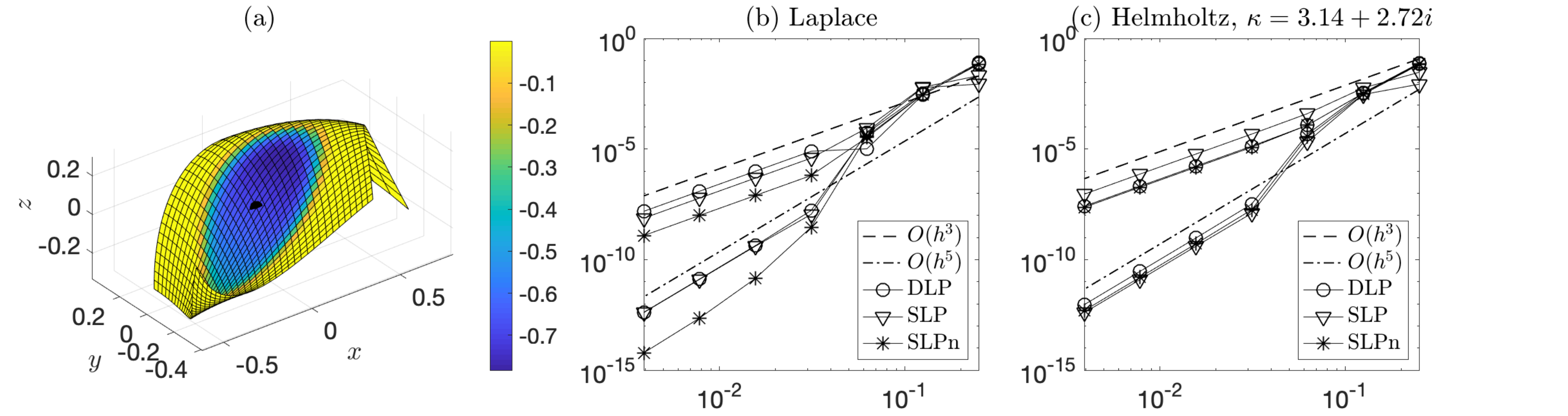}
\caption{(a) Use the $O(h^3)$ and $O(h^5)$ corrected trapezoidal rules to integrate the Laplace and Helmholtz layer potentials on a quartic surface patch with a random density function. The singular point is shown as the black dot at the center of the patch. Color represents the density $\sigma(u,v)$ defined by \eqref{eq:dens_func} with $a=.809, b=-.221$. Convergence against grid size $h$ are shown in (b) for Laplace and (c) for Helmholtz with wavenumber $\kappa$. Circle, triangle and asterisk correspond to SLP, DLP and the normal derivative of SLP (denoted ``SLPn''), respectively.}\label{fig:convergence_patch}
\end{figure}

\paragraph{Example 2.}

We solve the Dirichlet and Neumann boundary value problems (BVPs) associated with the Laplace and Helmholtz equations exterior to a toroidal surface. The Laplace BVPs are reformulated as second kind integral equations based on potential theory \cite[\S 6.4]{kress2014linear} and the Helmholtz BVPs as combined-field integral equations \cite[\S 2]{bremer2015high}, then these integral equations are discretized into the form \eqref{eq:nystromlinsys} using the Nystr\"om method with the $O(h^5)$ corrected trapezoidal rules shown in Example 1. The solution procedure is as follows: first the quadrature correction weights are pre-computed, then the integral equations are solved iteratively using GMRES with a tolerance $\epsilon_\text{\tiny GMRES}$; in each iteration, we first apply the punctured trapezoidal rule discretization of the Laplace/Helmholtz kernel using the Fast Multiple Method (FMM), then the pre-computed local correction weights are applied. The overall computational cost is $O(N)$ when $N$ discretization points are used.

Figure \ref{fig:laplace_bvp_O(h^5)} and Table \ref{tbl:corrected_bvps_stats} show the convergence in relative sup-norm error and the timing results. The $O(h^5)$ scaling in the errors and $O(N)$ scaling in time are clearly observed. Notice that for the Helmholtz BVPs with $\kappa=25+i$ (meaning the geometry has a diameter of about 12 wavelengths), the errors $E^\text{diri}$ and $E^\text{neu}$ are higher than the other cases since the solution is more oscillatory hence requiring higher resolutions to resolve. The number of GMRES iterations $N_\mathrm{iter}$ has an upper bound independent of the number of points $N$, this demonstrates that under the corrected trapezoidal rule discretization, the well-conditioning of the second kind integral operators is preserved. 

\begin{figure}[htbp]
\centering
\includegraphics[width=0.28\textwidth]{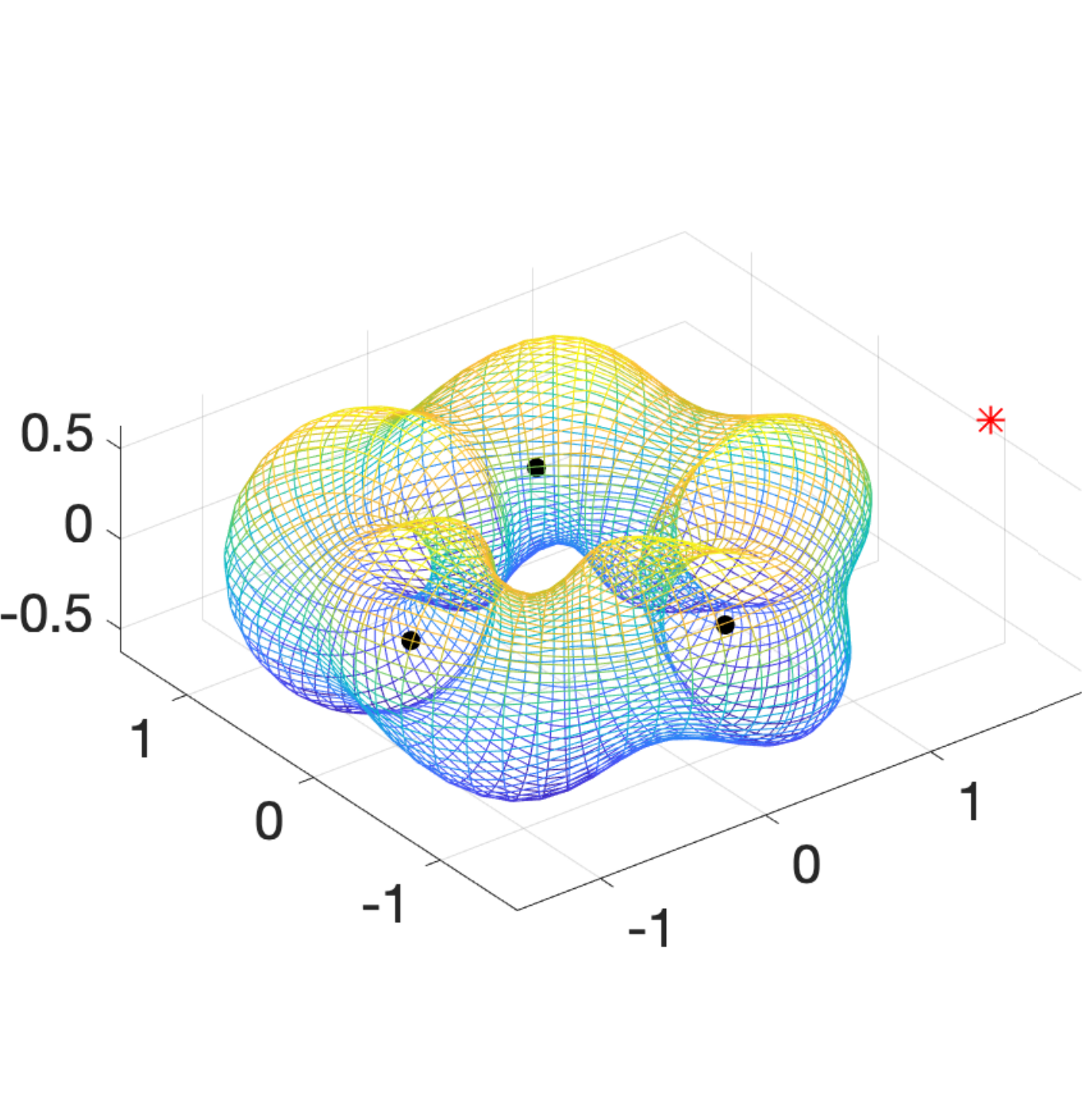} 
\includegraphics[width=0.7\textwidth]{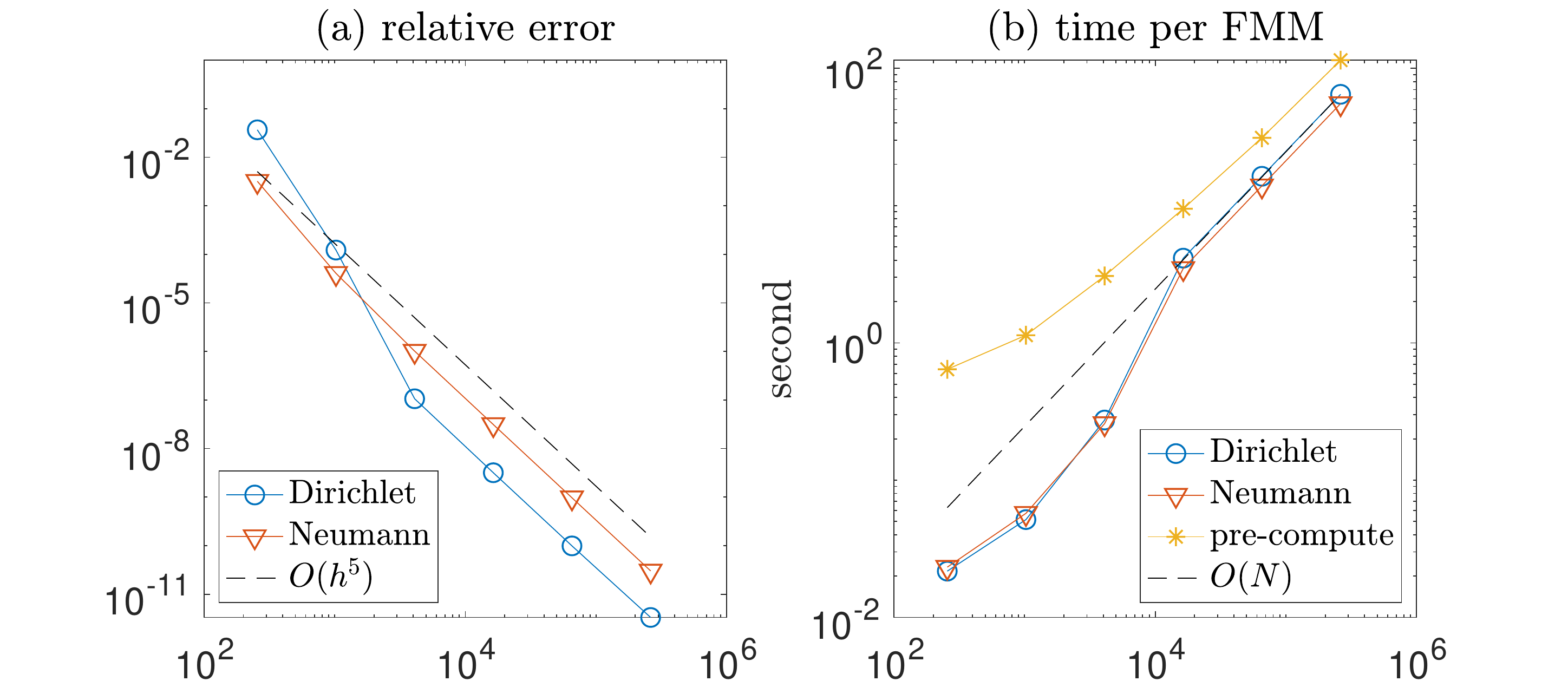}
\caption{Solution of the Laplace boundary value problems exterior to a toroidal surface discretized using the $O(h^5)$ corrected trapezoidal rule. The exact solution is generated by point sources located inside the surface shown by the block dots, the accuracy of solution is verified at a point away from the surface shown by the red asterisk. A total number of $N$ points are used so that grid spacing is $h=O(N^{-\frac{1}{2}})$. (a) $O(h^5)$ convergence of the relative error for both BVPs. (b) Average times per GMRES-iteration via FMM for both BVPs as well as the pre-computation times of the correction weights $\tau_{\mu,\nu}$. Both horizontal axes are $N$.}
\label{fig:laplace_bvp_O(h^5)}
\end{figure}

\begin{table}[htbp]
    \centering
    \def\arraystretch{1.2}
    \begin{tabular}{c|c|cccccc}
        \multicolumn{2}{c}{Laplace} & \multicolumn{6}{|c}{$\epsilon_\text{\tiny GMRES}=10^{-12}$}\\
        \hline
        $N$ & $T_\text{wei}$ & $T_\text{iter}^\text{diri}$ & $N_\text{iter}^\text{diri}$ & $T_\text{iter}^\text{neu}$ & $N_\text{iter}^\text{neu}$ & $E^\text{diri}$ & $E^\text{neu}$ \\
        \hline
        256   & 0.6 & 0.02 & 29 & 0.02 & 21 & 3.7e-02 & 3.2e-03\\
        1024  & 1.1 & 0.05 & 26 & 0.06 & 20 & 1.2e-04 & 4.1e-05\\
        4096  & 3.0 & 0.27 & 25 & 0.26 & 19 & 1.1e-07 & 1.0e-06\\
        16384 & 9.4 & 4.1  & 22 & 3.5  & 19 & 3.2e-09 & 3.1e-08\\
        65536 & 31.1 & 16.3 & 22& 13.9 & 19 & 9.9e-11 & 9.7e-10\\
        262144& 114.6& 64.6 & 22& 55.3 & 19 & 3.3e-12 & 3.0e-11\\
        \hline
        \multicolumn{2}{c}{Helmholtz} & \multicolumn{6}{|c}{$\kappa=1.42 + 1.11i$, $\epsilon_\text{\tiny GMRES}=10^{-12}$}\\
        \hline
        $N$ & $T_\text{wei}$ & $T_\text{iter}^\text{diri}$ & $N_\text{iter}^\text{diri}$ & $T_\text{iter}^\text{neu}$ & $N_\text{iter}^\text{neu}$ & $E^\text{diri}$ & $E^\text{neu}$ \\
        \hline
        256   & 0.73 & 0.02 & 29 & 0.04 & 21 & 9.3e-03 & 1.3e-02\\
        1024  & 1.23 & 0.21 & 23 & 0.32 & 18 & 3.6e-05 & 8.3e-05\\
        4096  & 3.4 & 2.4  & 21 & 3.2  & 18 & 1.6e-06 & 2.2e-06\\
        16384 & 10.5 & 16.6 & 21 & 19.9 & 18 & 5.1e-08 & 7.0e-08\\
        65536 & 34.1 & 68.4 & 21 & 82.0 & 18 & 1.6e-09 & 2.2e-09\\
        262144& 127.0 &248.9 & 21 & 293.8 & 18 & 5.1e-11 & 6.8e-11\\
        \hline
        \multicolumn{2}{c}{Helmholtz} & \multicolumn{6}{|c}{$\kappa=25 + 1i$ ($12$ wavelengths), $\epsilon_\text{\tiny GMRES}=10^{-9}$}\\
        \hline
        $N$ & $T_\text{wei}$ & $T_\text{iter}^\text{diri}$ & $N_\text{iter}^\text{diri}$ & $T_\text{iter}^\text{neu}$ & $N_\text{iter}^\text{neu}$ & $E^\text{diri}$ & $E^\text{neu}$ \\
        \hline
        256   & 1.3 & 0.05 & 40 & 0.09 & 40 & 1.1e+00 & 2.5e+01\\
        1024  & 1.2 & 0.28 & 40 & 0.43 & 33 & 1.3e+00 & 3.0e+00\\
        4096  & 3.4 & 4.2  & 36 & 5.0  & 22 & 9.7e-02 & 4.1e-01\\
        16384 & 10.2 & 14.8 & 29 & 17.1 & 22 & 1.8e-03 & 2.7e-03\\
        65536 & 34.7 & 58.7 & 29 & 70.0 & 22 & 5.1e-05 & 7.5e-05\\
        262144& 126.6 &204.6 & 29 & 245.7 & 22 & 1.6e-06 & 2.3e-06
    \end{tabular}
    \caption{Timings and convergence of solving the Dirichlet and Neumann boundary value problems associated with the Laplace and Helmholtz equations exterior to a toroidal surface (see Figure \ref{fig:laplace_bvp_O(h^5)}). The superscripts ``diri'' and ``neu'' indicates Dirichlet and Neumann problems, respectively. $T_\text{wei}$ is the time for pre-computing the $O(h^5)$ quadrature correction weights $\tau_{\mu,\nu}$ for all the involved integral operators: the DLP, SLP and its normal gradient. $N_\text{iter}$ is the number of GMRES iteration and $T_\text{iter}$ is the average time of applying the punctured trapezoidal rule via FMM in each iteration (the time to apply the correction weights is negligible). $E$ in the last two columns denotes the relative $\infty$-norm error.}
    \label{tbl:corrected_bvps_stats}
\end{table}

\paragraph{Example 3.} Finally, we visualize the Epstein zeta function.

The left panel of Figure \ref{fig:epstein} shows the absolute value of $Z_A(s)$ evaluated using the formula \eqref{eq:z(s)} with fixed $E,F,G$ and varying $s$. This picture should look familiar to those who are familiar with the Riemann $\zeta(s)$. Clearly, $Z_A(s)$ is smooth but with a pole at $s=2$.

The right panel of Figure \ref{fig:epstein} shows the scaled value of the Epstein zeta function (evaluated using \eqref{eq:epstein_anal_s1_trunc})
$$
\sqrt{E}\cdot Z_A(1) \;\dot{=}\; \sideset{}'\sum\frac{1}{\sqrt{i^2+2\beta ij+\alpha j^2}}
$$
as a function of two variables $\alpha=\frac{G}{E},\,\beta=\frac{F}{E}$. The dot above the equal sign indicates that the above expression is understood as in the sense of analytic continuation. The function is singular near the parabola $\alpha=\beta^2$, which corresponds to $EF-G^2=0$, i.e. the local basis $\rr_u$ and $\rr_v$ are linearly dependent; away from the curve $\alpha=\beta^2$, the graph is smooth and well-behaved, which implies that given a regular surface with a reasonable parameterization, the correction weights as Epstein zeta function values converge uniformly rapidly.

\begin{figure}[htbp]
\centering
\includegraphics[width=0.45\textwidth]{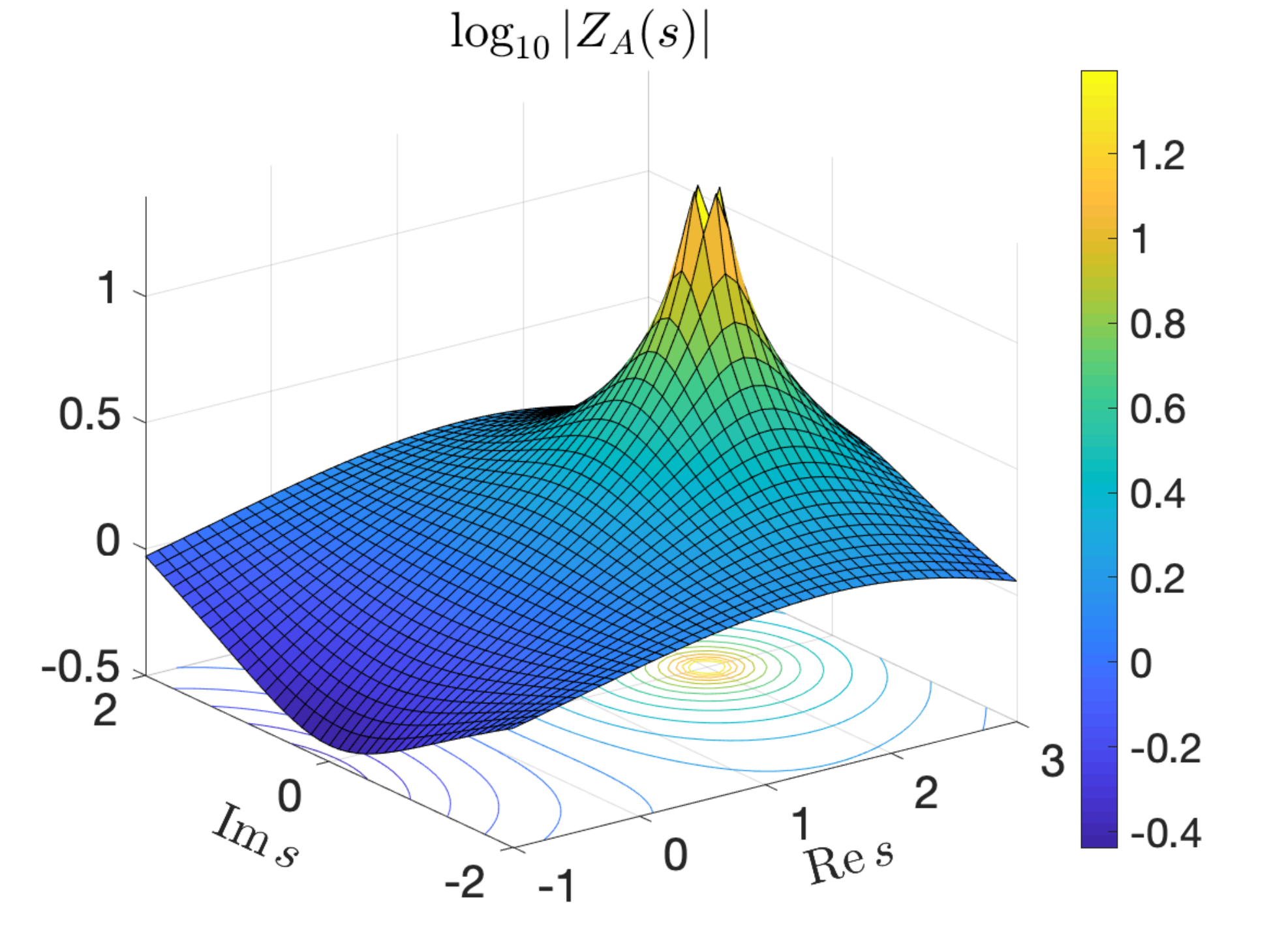} \includegraphics[width=0.45\textwidth]{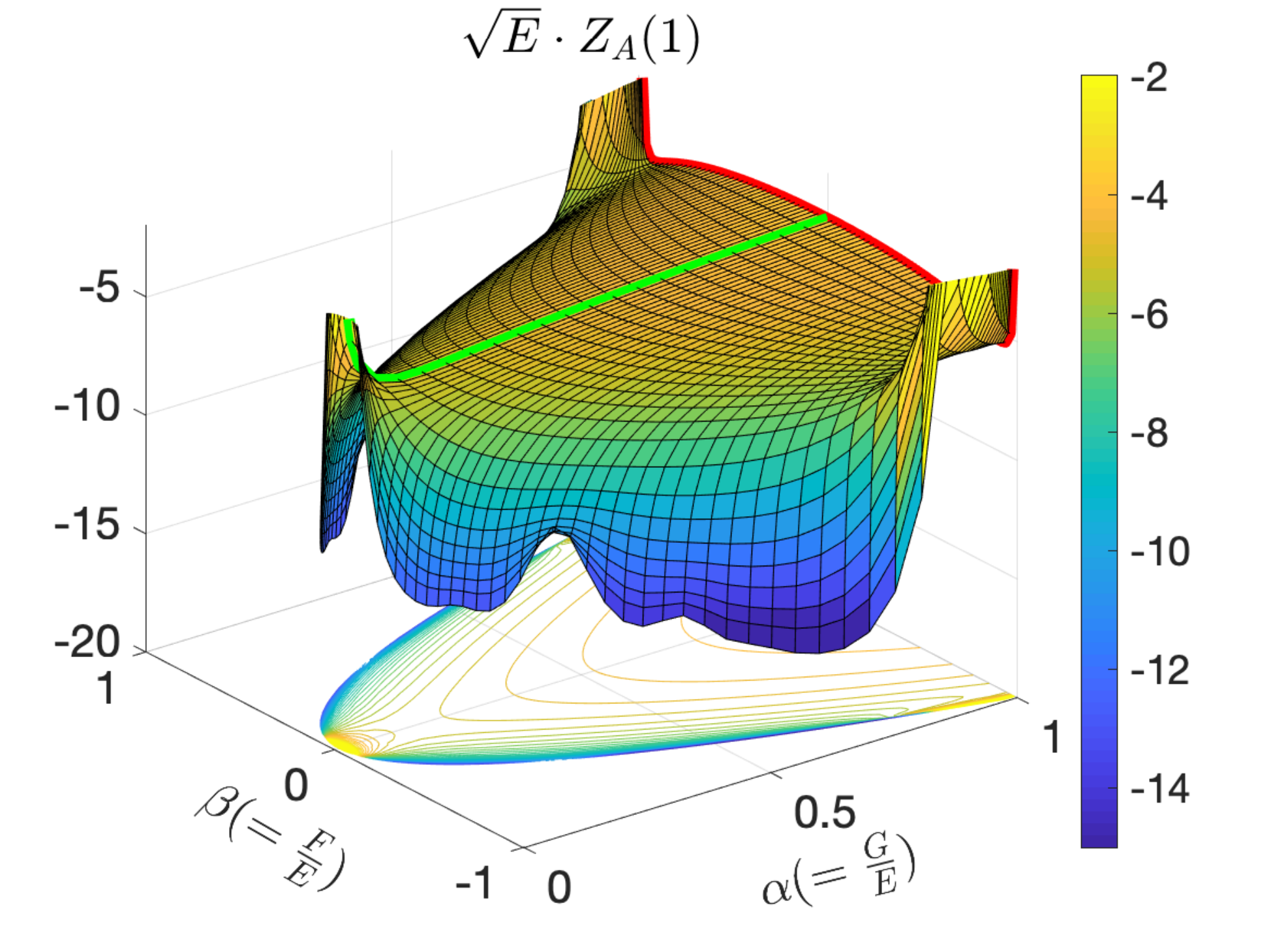}
\caption{Left: log-magnitude of $Z_A(s)$ for $E = 3.1, F = 0.8, G = 2.3.$ Clearly $Z_A(s)$ has a pole at $s=2$. Right: $Z_{A}(1)$ scaled by $\sqrt{E}$ for $(E,F,G)=(1,\beta ,\alpha)\cdot E$. Observations: (1) the constraint $EG-F^2> 0$ corresponds to $\alpha>\beta^2$, and the value of $Z_A$ diverges as $\alpha-\beta^2\to0$, (2) the values on the line $\beta=0$ (green curve) corresponds to $F=0$, i.e. $\rr_u\perp\rr_v$, (3) the values on the line $\alpha=1$ (red curve) corresponds to $E=G$, i.e. $\rr_u$ and $\rr_v$ have the same length but form different angles.}\label{fig:epstein}
\end{figure}

\section{Conclusion}

We have presented a new systematic and efficient approach to constructing locally corrected trapezoidal rules for a class of weakly singular boundary integral operators on curved surfaces in three dimension. In particular, we demonstrate that for kernels involving a singularity of strength $O(|\xx|^{-1})$, the correction weights can be explicitly computed with the help of the Epstein zeta function and its parametric derivatives. Complete error analysis of the quadratures is presented.  The $O(h^3)$ and $O(h^5)$ corrected quadratures are implemented. In particular, the $O(h^3)$ quadratures are simple and efficient to construct and can easily extend to many common Fredholm operators. We have shown that these quadratures are highly compatible with fast algorithms such as the FMM, achieving an $O(N)$ overall computational complexity when solving Laplace or Helmholtz BVPs on a toroidal surface. 

Codes for Examples 1 and 2 (not including FMM) in Section \ref{sc:results} are available at \url{https://github.com/bobbielf2/ZetaTrap3D}. 

One limitation of our method is that higher derivatives of the geometry are required for high-order corrections. These can in practice be obtained through numerical differentiation, although this is likely to degrade the accuracy unless high quality parameterizations of the surface are available.

The fact that the quadrature corrections are strongly localized makes them particularly well suited for use with fast direct solvers, such as those described in \cite{2019_martinsson_book}. Work in this direction is under way, as is an extension to other integral operators such as, e.g., Stokes kernels used to model fluid problems as in \cite{marple2016fast}. We are also investigating whether partition of unity domain decomposition techniques as in \cite{bruno2001fast} can help in generalizing these ideas to more general geometries. 

\begin{acknowledgements}
The work reported was supported by the Office of Naval Research
(grant N00014-18-1-2354) and by the National Science Foundation (grant DMS-1620472).
\end{acknowledgements}

%
%

\appendix

\section{Wigner-type limit formulae for the Laplace SLP}
\label{sc:app:slp_wigner-type}
To complete the expressions \eqref{eq:D0-D5_lap_slp}, we give formulae for $\cC_1,\cC_2$ and $\cC_{01}$. Firstly, $\cC_1$ and $\cC_2$ are the Wigner-type limits associated with $\cW^{(2)}$ defined by \eqref{eq:w2_wigner-type}, which we restate here
\begin{equation}
\cW^{(2)}(\uu) = -\frac{1}{2}\frac{2\dd^1\cdot\dd^2}{Q^{\frac{3}{2}}},
\end{equation}
where the numerator is a homogeneous polynomial:
$$
2\dd^1\cdot\dd^2 = L_{30}^Au^3+L_{21}^Au^2v+L_{12}^Auv^2+L_{03}^Av^3,
$$
such that
$$
L_{30}^A = \rr_u\cdot\rr_{uu},\quad 
L_{21}^A = 2\rr_u\cdot\rr_{uv}+\rr_v\cdot\rr_{uu},\quad L_{12}^A = 2\rr_{uv}\cdot\rr_v+\rr_u\cdot\rr_{vv},\quad  L_{03}^A = \rr_v\cdot\rr_{vv},
$$
where all the involved derivatives are evaluated at $\mathbf{0}.$ Thus we have the following Wigner-type limits
\begin{align}
    \cC_1 = -\cL[\cW^{(2)}(\uu)\,u] = -2L^{(A)}_{3,1}\cdot\square^{(2)}Z_A(-1)
\label{eq:w2_wigner-type_C1}\\
    \cC_2 = -\cL[\cW^{(2)}(\uu)\,v] = -2L^{(A)}_{3,2}\cdot\square^{(2)}Z_A(-1)
\label{eq:w2_wigner-type_C2}
\end{align}
where
\begin{equation}
    \square^{(2)} :=\left(\frac{\pd^2}{\pd E^2},\frac{1}{2}\frac{\pd^2}{\pd E\pd F},\frac{1}{4}\frac{\pd^2}{\pd F^2}, \frac{1}{2}\frac{\pd^2}{\pd G\pd F}, \frac{\pd^2}{\pd G^2}\right)
    \label{eq:square(2)}
\end{equation}
is a vector-valued second-derivative operator, and where the constants in the vectors $L^{(A)}_{3,1}$ and $L^{(A)}_{3,2}$ correspond to the terms $u\,\cW^{(2)}(\uu)$ and $v\,\cW^{(2)}(\uu)$, respectively, therefore
$$
    \begin{aligned}
    L^{(A)}_3 &:= (L_{30}^A,L_{21}^A,L_{12}^A,L_{03}^A),\\
    L^{(A)}_{3,1} &:= (1,0)*L^{(A)}_3 = (L_{30}^A,L_{21}^A,L_{12}^A,L_{03}^A,0),\\
    L^{(A)}_{3,2} &:= (0,1)*L^{(A)}_3 = (0,L_{30}^A,L_{21}^A,L_{12}^A,L_{03}^A).
    \end{aligned}
$$
where ``$*$'' represents convolution (given that these are coefficients of the product of polynomials). Next, $\cC_{01}$ is the Wigner-type limit associated with the function $\cW^{(3)}$, given by
$$
\cW^{(3)}(\uu) = -\frac{1}{2}\frac{2\dd^1\cdot\dd^3+\dd^2\cdot\dd^2}{Q^{\frac{3}{2}}}+\frac{3}{8}\frac{(2\dd^1\cdot\dd^2)^2}{Q^{\frac{5}{2}}},
$$
then, expanding the denominator as above, one derives that
\begin{equation}
    \cC_{01} = -\cL[\cW^{(3)}(\uu)] = -\big(2L_4^{(A)}\cdot\square^{(2)}+L_6^{(A)}\cdot\square^{(3)}\big)Z_A(-1)
\label{eq:w3_wigner-type_C01}
\end{equation}
where
\begin{equation}
    \square^{(3)} =\Bigg(\frac{\pd^3}{\pd E^3},\frac{1}{2}\frac{\pd^3}{\pd E^2\pd F},\frac{1}{4}\frac{\pd^3}{\pd E\pd F^2},\frac{1}{8}\frac{\pd^3}{\pd F^3},\frac{1}{4}\frac{\pd^3}{\pd F^2\pd G},  \frac{1}{2}\frac{\pd^3}{\pd F\pd G^2}, \frac{\pd^3}{\pd G^3}\Bigg)
    \label{eq:square(3)}
\end{equation}
is a vector-valued third-derivative operator, and where the vectors $L^{(A)}_4$ and $L^{(A)}_6$ correspond to the expansions of $(2\dd^1\cdot\dd^3+\dd^2\cdot\dd^2)$ and $(2\dd^1\cdot\dd^2)^2$, respectively, such that
$$
\begin{aligned}
L^{(A)}_4 &= (L_{40}^A,L_{31}^A,L_{22}^A,L_{13}^A,L_{04}^A),\\
L^{(A)}_6 &= L^{(A)}_3*L^{(A)}_3,
\end{aligned}
$$
where
\begin{align*}
L_{40}^A &= \rr_u\cdot\rr_{uuu}/3+\rr_{uu}\cdot\rr_{uu}/4,\\
L_{31}^A &= \rr_v\cdot\rr_{uuu}/3+\rr_u\cdot\rr_{uuv}+\rr_{uu}\cdot\rr_{uv},\\
L_{22}^A &= \rr_v\cdot\rr_{uuv}+\rr_u\cdot\rr_{uvv}+\rr_{uu}\cdot\rr_{vv}/2+\rr_{uv}\cdot\rr_{uv},\\
L_{13}^A &= \rr_v\cdot\rr_{uvv}+\rr_u\cdot\rr_{vvv}/3+\rr_{uv}\cdot\rr_{vv},\\
L_{04}^A &= \rr_v\cdot\rr_{vvv}/3+\rr_{vv}\cdot\rr_{vv}/4.
\end{align*}

\section{\texorpdfstring{$O(h^5)$}{TEXT} corrected trapezoidal rule for the Laplace DLP}
\label{sc:app:dlp_corr_h5}
Consider the DLP kernel $\cK(\uu)$ given by
$$\cK(\uu) = \frac{\big(\rr(\mathbf{0})-\rr(\uu)\big)\cdot(\rr_u(\uu)\times\rr_v(\uu))}{r(\uu)^3}=-\frac{\rr(\uu)\cdot(\rr_u(\uu)\times\rr_v(\uu))}{r(\uu)^3},$$
then the $O(h^5)$ quadrature for the DLP is given by
\begin{equation}
    \int_{|\uu|\leq a}\cK(\uu)\sigma(\uu)\,\d\uu = \sideset{}'\sum_{|\ii |\leq N}\cK(\ii h)\sigma(\ii h) + \sum_{(\mu,\nu)\in U_2}\sigma(\mu h,\nu h)\tau_{\mu,\nu} +O(h^5)
\label{eq:corr_trap_rule_dlp}
\end{equation}
where $\sigma$ is a smooth density function and $U_2$ is the 9-point stencil as defined in \eqref{eq:stencil_curved}. The correction weights satisfy the system \eqref{eq:corr_weights_system} but with the quantities $D_0$--$D_5$ replaced by \eqref{eq:D0-D5_lap_dlp} below, which we will derive next.

To derive the formulae for the correction weights, we follow an analysis similar to Section \ref{sc:high-order_analysis}. First the DLP kernel $\cK(\uu)$ is expanded using the binomial series
\begin{equation}
(1+x)^{-\frac{3}{2}} = 1-\frac{3}{2}x+\frac{15}{8}x^2+\dots+\binom{-\frac{3}{2}}{m}x^m+\dots
\end{equation}
thus
\begin{equation}
-\frac{\rr\cdot(\rr_u\times\rr_v)}{r^3}= -\rr\cdot(\rr_u\times\rr_v)\left(\frac{1}{Q^{\frac{3}{2}}}-\frac{3(r^2-Q)}{2Q^\frac{5}{2}}+\frac{15(r^2-Q)^2}{8Q^{\frac{7}{2}}}+\dots\right)
\label{eq:dlp_expan_in_1st_fund_form}
\end{equation}
Then the numerators $O((r^2-Q)^m)$ are further expanded as in \eqref{eq:(r2-Q)^m}; the term $-\rr\cdot(\rr_u\times\rr_v)$ can be expanded as
$$
\begin{aligned}
-\rr\cdot(\rr_u\times\rr_v) &= -(\dd^1+\dd^2+\dd^3+\dd^4+O(\uu^5))\cdot\Big((\dd^1_u+\dd^2_u+\dd^3_u+\dd^4_u+O(\uu^4))\times\\ &\hspace*{2cm}(\dd^1_v+\dd^2_v+\dd^3_v+\dd^4_v+O(\uu^4))\Big)\\
&=q_2^{(B)}+q_3^{(B)}+q_4^{(B)}+O(\uu^5)
\end{aligned}
$$
where
$$
\begin{aligned}
q_2^{(B)}(\uu) &= \frac{1}{2}\big(L^B_{20}u^2+L^B_{11}uv+L^B_{02}v^2\big)\\
q_3^{(B)}(\uu) &=  \frac{1}{2}(L^B_{30}u^3+L^B_{21}u^2v+L^B_{12}uv^2+L^B_{03}v^3)\\
q_4^{(B)}(\uu) &= \frac{1}{4}(L^B_{40}u^4+L^B_{31}u^3v+L^B_{22}u^2v^2+L^B_{13}uv^3+L^B_{04}v^4)
\end{aligned}
$$
where the involved constants are defined as follows (all involved derivatives are evaluated at $\mathbf{0}$),
$$
\begin{aligned}
L^B_{20} &= (\rr_u\times\rr_v)\cdot\rr_{uu},\quad L^B_{11} = 2(\rr_u\times\rr_v)\cdot\rr_{uv},\quad L^B_{02} = (\rr_u\times\rr_v)\cdot\rr_{vv},\\
L^B_{30} &= \frac{2}{3}(\rr_u\times\rr_v)\cdot\rr_{uuu}-(\rr_u\times\rr_{uu})\cdot\rr_{uv}\\
L^B_{21} &=2(\rr_u\times\rr_v)\cdot\rr_{uuv}-(\rr_u\times\rr_{uu})\cdot\rr_{vv}-(\rr_v\times\rr_{uu})\cdot\rr_{uv}\\
L^B_{12} &=2(\rr_u\times\rr_v)\cdot\rr_{uvv}-(\rr_u\times\rr_{uv})\cdot\rr_{vv}-(\rr_v\times\rr_{uu})\cdot\rr_{vv}\\
L^B_{03} &= \frac{2}{3}(\rr_u\times\rr_v)\cdot\rr_{vvv}-(\rr_v\times\rr_{uv})\cdot\rr_{vv}
\end{aligned}
$$

$$
\begin{aligned}
L^B_{40} &= \frac{1}{6}\Big(-6\rr_u\times \rr_{uu}\cdot\rr_{uuv}+8\rr_u\times \rr_{uv}\cdot\rr_{uuu}+3\rr_u\times \rr_v\cdot\rr_{uuuu}-2\rr_v\times\rr_{uu}\cdot \rr_{uuu}\Big)\\
L^B_{31} &= \frac{2}{3}\Big(-3\rr_u\times \rr_{uu}\cdot\rr_{uvv}+2\rr_u\times \rr_{vv}\cdot\rr_{uuu}+3\rr_u\times \rr_v\cdot\rr_{uuuv}\\
         &\qquad+3\rr_u\times \rr_{uv}\cdot\rr_{uuv}-3\rr_v\times \rr_{uu}\cdot\rr_{uuv}+\rr_v\times \rr_{uv}\cdot\rr_{uuu}\Big)\\
\
L^B_{22} &= \Big(-\rr_u\times \rr_{uu}\cdot\rr_{vvv}+3\rr_u\times \rr_{vv}\cdot\rr_{uuv}\\
        &\qquad +3\rr_u\times \rr_v\cdot\rr_{uuvv}-3\rr_v\times\rr_{uu} \cdot\rr_{uvv}+\rr_v\times \rr_{vv}\cdot\rr_{uuu}\Big)\\
L^B_{13} &= \frac{2}{3}\Big(-\rr_u\times \rr_{uv}\cdot\rr_{vvv}+3\rr_u\times \rr_{vv}\cdot\rr_{uvv}+3\rr_u\times \rr_v\cdot\rr_{uvvv}\\
          &\qquad-2\rr_v\times \rr_{uu}\cdot\rr_{vvv}+3\rr_v\times \rr_{vv}\cdot\rr_{uuv}-3\rr_v\times \rr_{uv}\cdot\rr_{uvv}\Big)\\
L^B_{04} &= \frac{1}{6}\Big(3\rr_u\times \rr_v\cdot\rr_{vvvv}+2\rr_u\times \rr_{vv}\cdot\rr_{vvv}-8\rr_v\times \rr_{uv}\cdot\rr_{vvv}+6\rr_v\times\rr_{vv}\cdot\rr_{uvv}\Big)
\end{aligned}
$$

Similar to Table \ref{tbl:slp_singularity_summary}, we summarize in Table \ref{tbl:dlp_singularity_summary} the singularity information associated with the DLP kernel expansion \eqref{eq:dlp_expan_in_1st_fund_form}.

\begin{table}[htbp]
\centering
\begin{tabular}{l | lll | l}
$\frac{-\rr\cdot(\rr_u\times\rr_v)}{r^3}$ & $\frac{-\rr\cdot(\rr_u\times\rr_v)}{\sqrt{Q}^3}$ &  $-\frac{3}{2}\frac{-\rr\cdot(\rr_u\times\rr_v)(r^2-Q)}{\sqrt{Q}^5}$ & $\frac{15}{8}\frac{-\rr\cdot(\rr_u\times\rr_v)(r^2-Q)^2}{\sqrt{Q}^7}$ & Density \\[10pt]
\hline
$\cV^{(1)}$ & $\frac{q_2^{(B)}}{\sqrt{Q}^3}$ &  & & $\Theta(\uu^2), \Theta(\uu^0)$  \\[15pt]
$\cV^{(2)}$ & $\frac{q_3^{(B)}}{\sqrt{Q}^3}$  & $-\frac{3}{2}\frac{q_3^{(1)}q_2^{(B)}}{\sqrt{Q}^5}$  &  &  $\Theta(\uu^1)$ \\[15pt]
$\cV^{(3)}$ & $\frac{q_4^{(B)}}{\sqrt{Q}^3}$ & $-\frac{3}{2}\frac{q_4^{(1)}{q_2^{(B)}}+q_3^{(1)}q_3^{(B)}}{\sqrt{Q}^5}$  & $\frac{15}{8}\frac{q_6^{(2)}q_2^{(B)}}{\sqrt{Q}^7}$  & $\Theta(\uu^0)$
\end{tabular}
\caption{Summary of the functions whose associated Wigner-type limits are helpful for constructing the $O(h^5)$ corrected quadrature for the Laplace DLP. See Table \ref{tbl:slp_singularity_summary} for how this table is organized. $q_k^{(m)}$ are as defined in \eqref{eq:(r2-Q)^m}.}\label{tbl:dlp_singularity_summary}
\end{table}

Similar to \eqref{eq:D0-D5_lap_slp}, we define the following quantities based on Table \ref{tbl:dlp_singularity_summary} which are useful for constructing the $O(h^5)$ quadrature \eqref{eq:corr_trap_rule_dlp} for the DLP:
\begin{equation}
\begin{aligned}
D_0 &= -\cL[\cV^{(1)}(\uu)]\,h-\cL[\cV^{(3)}(\uu)]\,h^3 = L^{(B)}_2\cdot\square^{(1)}Z(1)\,h + \cC_{01}\,h^3\\
D_1 &= -\cL[\cV^{(2)}(\uu)\,u]\,h^2 = \cC_{1}\,h^2\\
D_2 &= -\cL[\cV^{(2)}(\uu)\,v]\,h^2 = \cC_{2}\,h^2\\
D_3 &= -\cL[\cV^{(1)}(\uu)\,u^2]\,h = 2\big((1,0,0)*L^{(B)}_2\big)\cdot\square^{(2)}\,Z(-1)\,h\\
D_4 &= -\cL[\cV^{(1)}(\uu)\,v^2]\,h = 2\big((0,0,1)*L^{(B)}_2\big)\cdot\square^{(2)}\,Z(-1)\,h\\
D_5 &= -\cL[\cV^{(1)}(\uu)\,uv]\,h = 2\big((0,1,0)*L^{(B)}_2\big)\cdot\square^{(2)}\,Z(-1)\,h
\end{aligned}
    \label{eq:D0-D5_lap_dlp}
\end{equation}
where
$$
\begin{aligned}
\cC_{01} &=  \Big((L^{(B)}_4\cdot\square^{(2)}) + 2(L^{(B)}_6\cdot\square^{(3)})+  (L^{(A)}_6* L^{(B)}_2\cdot\square^{(4)})\Big)Z(-1)\\
\cC_{1} &= 2\Big((1,0)*L^{(B)}_3\cdot\square^{(2)}+(1,0) * L^{(A)}_3 * L^{(B)}_2\cdot\square^{(3)}\Big)Z(-1)\\
\cC_{2} &= 2\Big((0,1)*L^{(B)}_3\cdot\square^{(2)}+(0,1) * L^{(A)}_3 * L^{(B)}_2\cdot\square^{(3)}\Big)Z(-1).
\end{aligned}
$$
The involved vector-valued differential operators, besides (\ref{eq:square(2)},~\ref{eq:square(3)}), are defined as
\begin{equation}
\begin{aligned}
    \square^{(1)} &:= \left(\frac{\partial}{\partial E},\frac{1}{2}\frac{\partial}{\partial F},\frac{\partial}{\partial G}\right),\\
    \square^{(4)} &:= \bigg(\frac{\pd^4}{\pd E^4},\frac{1}{2}\frac{\pd^4}{\pd E^3\pd F},\frac{1}{4}\frac{\pd^4}{\pd E^2\pd F^2}, \frac{1}{8}\frac{\pd^4}{\pd E\pd F^3},\\
    & \qquad \frac{1}{16}\frac{\pd^4}{\pd F^4}, \frac{1}{8}\frac{\pd^4}{\pd F^3\pd G}, \frac{1}{4}\frac{\pd^4}{\pd F^2\pd G^2}, \frac{1}{2}\frac{\pd^4}{\pd F\pd G^3}, \frac{\pd^4}{\pd G^4}\bigg),
\end{aligned}
\label{eq:square(1,4)}
\end{equation}
and the involved constant vectors are
$$
\begin{aligned}
L^{(B)}_2 &:= (L^B_{20}, L^B_{11}, L^B_{02}),\\
L^{(B)}_3 &:= (L^B_{30}, L^B_{21}, L^B_{12}, L^B_{03})\\
L^{(B)}_4 &:= (L^B_{40}, L^B_{31}, L^B_{22}, L^B_{13}, L^B_{04})\\
L^{(B)}_6 &:= L^{(A)}_4* L^{(B)}_2+L^{(A)}_3* L^{(B)}_3 
\end{aligned}
$$
and where $L^{(A)}_3,L^{(A)}_4$ and $L^{(A)}_6$ are as appeared in Appendix \ref{sc:app:slp_wigner-type} in the formulae for the SLP.

Substituting the above constants $D_0$--$D_5$ into the system \eqref{eq:corr_weights_system} results in the correction weights for the DLP in \eqref{eq:corr_trap_rule_dlp}.

\section{\texorpdfstring{$O(h^5)$}{TEXT} corrected trapezoidal rule for the normal gradient of Laplace SLP}
\label{sc:app:slpn_corr_h5}
We will omit the derivations in this section and just present the quadrature formula for the normal gradient of the Laplace SLP on a surface.

The kernel of the SLP normal gradient, centered at $\mathbf{0}$, is
$$\cK(\uu) = -\frac{\big(\rr(\mathbf{0})-\rr(\uu)\big)\cdot\nn(\mathbf{0})}{r(\uu)^3}=\frac{\rr(\uu)\cdot\nn(\mathbf{0})}{r(\uu)^3}.$$
then the associated $O(h^5)$ quadrature is given by
\begin{equation}
    \int_{|\uu|\leq a}\cK(\uu)P(\uu)\,\d\uu = \sideset{}'\sum_{|\ii |\leq N}\cK(\ii h)P(\ii h) + \sum_{(\mu,\nu)\in U_2}P(\mu h,\nu h)\tau_{\mu,\nu} +O(h^5)
\label{eq:corr_trap_rule_slpn}
\end{equation}
where $P(\uu)=\sigma(\uu)\,|\rr_u(\uu)\times\rr_v(\uu)|$ is smooth and $\sigma$ is a smooth density function, and where $U_2$ is the 9-point stencil as defined in \eqref{eq:stencil_curved}. The correction weights satisfy the system \eqref{eq:corr_weights_system} but with the quantities $D_0$--$D_5$ replaced by the following:
\begin{equation}
\begin{aligned}
D_0 &= -\cL[\cW_n^{(1)}(\uu)]\,h-\cL[\cW_n^{(3)}(\uu)]\,h^3 = L_C^{(2)}\cdot\square^{(1)}\,Z(1)\,h + \cC_{01}h^3\\
D_1 &= -\cL[\cW_n^{(2)}(\uu)\,u]\,h^2 = \cC_{1}\,h^2\\
D_2 &= -\cL[\cW_n^{(2)}(\uu)\,v]\,h^2 = \cC_{2}\,h^2\\
D_3 &= -\cL[\cW_n^{(1)}(\uu)\,u^2]\,h = 2\,(1,0,0)*L_C^{(2)}\cdot\square^{(2)}Z(-1)\,h\\
D_4 &= -\cL[\cW_n^{(1)}(\uu)\,v^2]\,h = 2\,(0,0,1)*L_C^{(2)}\cdot\square^{(2)}Z(-1)\,h\\
D_5 &= -\cL[\cW_n^{(1)}(\uu)\,uv]\,h = 2\,(0,1,0)*L_C^{(2)}\cdot\square^{(2)}Z(-1)\,h
\end{aligned}
\label{eq:D0-D5_lap_slpn}
\end{equation}
where
\begin{align*}
\cC_{01} &=  \Big(L^{(4)}_C\cdot\square^{(2)}+ 2\,L^{(6)}_C\cdot\square^{(3)}Z+ L^{(6)}_A* L^{(2)}_C\cdot\square^{(4)}\Big)Z(-1)\\
\cC_{1} &=  2\Big((1,0)*L^{(3)}_{C}\cdot\square^{(2)}+(1,0)*L^{(3)}_A * L^{(2)}_C\cdot\square^{(3)}\Big)Z(-1)\\
\cC_{2} &= 2\Big((0,1)*L^{(3)}_{C}\cdot\square^{(2)}+(0,1)*L^{(3)}_A * L^{(2)}_C\cdot\square^{(3)}\Big)Z(-1)
\end{align*}
such that
$$
\begin{aligned}
L^{(2)}_C &:= (\nn\cdot\rr_{uu},2\nn\cdot\rr_{uv},\nn\cdot\rr_{vv})\\ 
L^{(3)}_C &:= \left({\nn\cdot\rr_{uuu}}/{3},\nn\cdot\rr_{uuv}, \nn\cdot\rr_{uvv}, {\nn\cdot\rr_{vvv}}/{3}\right)\\
L^{(4)}_C &:= \left({\nn\cdot\rr_{uuuu}}/{6},{2\nn\cdot\rr_{uuuv}}/{3},\nn\cdot\rr_{uuvv},{2\nn\cdot\rr_{uvvv}}/{3},{\nn\cdot\rr_{vvvv}}/{6}\right)\\
L^{(6)}_C &:= L^{(4)}_A* L^{(2)}_C+L^{(3)}_A* L^{(3)}_C
\end{aligned}
$$
and where $L_A^{(3)}, L_A^{(4)}, L_A^{(6)}$ are as defined in Section \ref{sc:app:slp_wigner-type}, and $\square^{(k)}, k=1,2,3,4$, are as defined in (\ref{eq:square(2)},~\ref{eq:square(3)},~\ref{eq:square(1,4)}).

\section{\texorpdfstring{$O(h^5)$}{TEXT} corrected trapezoidal rules for the Helmholtz potentials}
\label{sc:app:helm_corr_h5}

The Helmholtz SLP, DLP, and the normal gradient of SLP (denoted SLPn) are related to the Laplace potentials by the following expansions:
\begin{align*}
\text{SLP:}\quad&\frac{e^{i\kappa r}}{r} = \frac{1}{r}+i\kappa-\frac{\kappa^2}{2}r+O(r^2)\\
\text{DLP:}\quad&\frac{e^{i\kappa r}(1-i\kappa r)}{r}\frac{(-\rr)\cdot(\rr_u\times\rr_v)}{r^2} = \left(\frac{1}{r}+\frac{\kappa^2}{2}r+O(r^2)\right)\frac{-\rr\cdot(\rr_u\times\rr_v)}{r^2}\\
\text{SLPn:}\quad&-\frac{e^{i\kappa r}(1-i\kappa r)}{r}\frac{(-\rr)\cdot\nn(\mathbf{0})}{r^2} = \left(\frac{1}{r}+\frac{\kappa^2}{2}r+O(r^2)\right)\frac{\rr\cdot\nn(\mathbf{0})}{r^2} 
\end{align*}
where the first term in each expansion is the corresponding Laplace kernel. Therefore based on the $O(h^5)$ quadratures for the Laplace potentials, one only needs to include the errors up to $\Theta(h^3)$ from the additional terms to construct the corresponding quadratures for the Helmholtz kernels. Specifically, these relevant additional terms are
\begin{align*}
\text{SLP:}\quad&i\kappa-\frac{\kappa^2}{2}r\\
\text{DLP:}\quad&-\frac{\kappa^2}{2}\frac{\rr\cdot(\rr_u\times\rr_v)}{r}\\
\text{SLPn:}\quad&\frac{\kappa^2}{2}\frac{\rr\cdot\nn(\mathbf{0})}{r}
\end{align*}
where the term $i\kappa$ is regular, all other terms associate to Wigner-type limits that are $O(h^3)$, hence only the leading errors need correction. By analyses similar to Section \ref{sc:high-order_quad}, quadratures for the Helmholtz potentials are constructed as follows:
\begin{itemize}
    \item The $O(h^5)$ quadrature for the Helmholtz SLP uses the same correction weights as the Laplace SLP constructed by \eqref{eq:D0-D5_lap_slp} and \eqref{eq:9pt_weights}, except that $D_0$ is modified as
    $$D_0\leftarrow D_0+\Big(i(\kappa h) + \frac{(\kappa h)^2}{2}Z(-1)\Big)\,h$$
    \item The $O(h^5)$ quadrature for the Helmholtz DLP uses the same correction weights as the Laplace DLP constructed by \eqref{eq:D0-D5_lap_dlp} and \eqref{eq:9pt_weights}, except that $D_0$ is modified as
    $$D_0\leftarrow D_0-\frac{(\kappa h)^2}{2}(L^{(B)}_2\cdot\square^{(1)})Z(-1)\,h$$
    \item The $O(h^5)$ quadrature for the normal gradient of the Helmholtz SLP uses the same correction weights as the corresponding Laplace potential constructed by \eqref{eq:D0-D5_lap_slpn} and \eqref{eq:9pt_weights}, except that $D_0$ is modified as
    $$D_0\leftarrow D_0-\frac{(\kappa h)^2}{2}(L^{(C)}_2\cdot\square^{(1)})Z(-1)\,h$$
\end{itemize}

\section{Formulae for \texorpdfstring{$Z(s)$}{TEXT} and its parametric derivatives}
\label{sc:app:epstein_higher_deriv}

To simplify notations, first define $s_1 :=\frac{s}{2}, s_2 := 1-s_1$ and the customed incomplete gamma function
$$g(s,x) := \int_1^\infty t^{s-1}e^{-\pi xt}\d t = \Gamma(s,\pi x)(\pi x)^{-s}$$
then
\begin{equation}
Z(s) = C_{s_1}(D)\left(-\frac{1}{s_2} - \frac{1}{s_1} + \sum_{i,j}{}'\Big(g(s_1, \tilde{Q}_A)+g(s_2, \tilde{Q}_A)\Big)\right).
\label{eq:z(s)}
\end{equation}
where 
$$
D = EG-F^2,\; C_{s}(D) :=  \frac{\pi^{s}}{\Gamma(s)\sqrt{D}^{s}},\; \tilde{Q}_A(u,v) = \frac{Q_A(u,v)}{\sqrt{D}} = \frac{Eu^2+2Fuv+Gv^2}{\sqrt{D}}.
$$

\subsection{Scalar derivatives of $Z(s)$}

We denote the scalar parametric $k$-th derivative of $Z(s)$ as
$$\square^kZ(s) = \square^kZ(s)\big|_{(L,M,N)} := \big(L\frac{\partial}{\partial E}+M\frac{\partial}{\partial F}+N\frac{\partial}{\partial G}\big)^kZ(s)$$

To derive the derivative formulae based on \eqref{eq:z(s)}, first note the following properties of $g(s,x)$
\begin{align*}
\square\,g(s,x) &= g(s+1,x)\cdot(-\pi\square x)\\
g(s+1,x) &= \frac{s\,g(s,x)+e^{-\pi x}}{\pi x},\quad g(s+2,x) = \frac{(s+1)\,g(s+1,x)+e^{-\pi x}}{\pi x}\\
g(-s_1,x) &= \frac{g(1-s_1,\pi \tilde{Q}_A)\cdot(\pi x) - e^{-\pi x}}{-s_1} =\frac{g(s_2,x)\cdot(\pi x) - e^{-\pi x}}{-s_1}
\end{align*}
these recurrence relations will allow one to evaluate all the derivatives of $Z(s)$ with only one gamma function evaluation at each point, greatly reducing the cost.

The formulae $\square^kZ(s)$ for $k\leq4$ are then given by
\begin{equation}
\square Z(s) = 
    -s_1H_D\,Z(s) -C_{s_1}(D)\sum_{i,j}{}'\Big(g(s_1+1, \tilde{Q}_A)+g(s_2+1, \tilde{Q}_A)\Big)\cdot\pi\square \tilde{Q}_A
\end{equation}
\begin{equation}
\begin{aligned}
\square^2Z(s) &= 
    -\left(s_1\square H_D+(s_1H_D)^2\right)Z(s)-2s_1H_D\,\square\,Z(s)\\
    &+C_{s_1}(D) \sum_{i,j}{}'\Big(g(s_1+2, \tilde{Q}_A)+g(s_2+2,\tilde{Q}_A)\Big)\cdot\left(\pi\square \tilde{Q}_A\right)^2\\
    &-C_{s_1}(D) \sum_{i,j}{}'\Big(g(s_1+1,\tilde{Q}_A)+g(s_2+1, \tilde{Q}_A)\Big)\cdot \pi\square^2 \tilde{Q}_A
\end{aligned}
\end{equation}
\begin{equation}
\begin{aligned}
\square^3Z(s) &=
    -\left(s_1\square^2 H_D+3s_1^2H_D\,\square\,H_D+(s_1H_D)^3\right)Z(s)\\
    &-\left(3s_1\square H_D+3(s_1H_D)^2\right)\square\,Z(s)-3s_1H_D\,\square^2\,Z(s)\\
    &-C_{s_1}(D) \sum_{i,j}{}'\Big(g(s_1+3, \tilde{Q}_A)+g(s_2+3,\tilde{Q}_A)\Big)\cdot\left(\pi\square \tilde{Q}_A\right)^3\\
    &+C_{s_1}(D) \sum_{i,j}{}'\Big(g(s_1+2, \tilde{Q}_A)+g(s_2+2,\tilde{Q}_A)\Big)\cdot3(\pi\square \tilde{Q}_A)(\pi\square^2\tilde{Q}_A)\\
    &-C_{s_1}(D) \sum_{i,j}{}'\Big(g(s_1+1,\tilde{Q}_A)+g(s_2+1, \tilde{Q}_A)\Big)\cdot \pi\square^3 \tilde{Q}_A
\end{aligned}
\end{equation}
\begin{equation}
\begin{aligned}
\square^4Z(s) &=
    -\left(s_1\square^3 H_D+3s_1^2(\square H_D)^2+4s_1^2H_D \square^2H_D+6s_1^3H_D^2\square H_D+(s_1H_D)^4\right)Z(s)\\
    &-\left(4s_1\square^2H_D+12s_1^2H_D\square H_D+4(s_1H_D)^3\right)\square Z(s)\\
    &-\left(6s_1\square H_D+6(s_1H_D)^2\right)\square^2Z(s)-4s_1H_D\,\square^3Z(s)\\
    &+C_{s_1}(D) \sum_{i,j}{}'\Big(g(s_1+4, \tilde{Q}_A)+g(s_2+4,\tilde{Q}_A)\Big)\cdot\left(\pi\square \tilde{Q}_A\right)^4\\
    &-C_{s_1}(D) \sum_{i,j}{}'\Big(g(s_1+3, \tilde{Q}_A)+g(s_2+3,\tilde{Q}_A)\Big)\cdot6(\pi\square \tilde{Q}_A)^2(\pi\square^2\tilde{Q}_A)\\
    &+C_{s_1}(D) \sum_{i,j}{}'\Big(g(s_1+2, \tilde{Q}_A)+g(s_2+2,\tilde{Q}_A)\Big)\cdot\Big(4(\pi\square \tilde{Q}_A)(\pi\square^3\tilde{Q}_A)+3(\pi\square^2\tilde{Q}_A)^2\Big)\\
    &-C_{s_1}(D) \sum_{i,j}{}'\Big(g(s_1+1,\tilde{Q}_A)+g(s_2+1, \tilde{Q}_A)\Big)\cdot \pi\square^4 \tilde{Q}_A
\end{aligned}
\end{equation}
To further reduce costs, recall for example the $O(h^3)$ errors used $Z(-1)$ while the $O(h)$ errors used $Z(1)$, thus one can compute $Z(s+2)$ along the way of computing $Z(s)$, with
\begin{equation*}
\begin{aligned}
Z(s+2) =& C_{s_1+1}(D)\left(\frac{1}{s_1} - \frac{1}{s_1+1} + \sum_{i,j}{}'\Big(g(s_1+1, \tilde{Q}_A)+g(-s_1, \tilde{Q}_A)\Big)\right),\\
\square Z(s+2) =& -(s_1+1)H_D\,Z(s+2) -C_{s_1+1}(D)\sum_{i,j}{}'\Big(g(s_1+2, \tilde{Q}_A)+g(s_2, \tilde{Q}_A)\Big)\cdot\pi\square \tilde{Q}_A,
\end{aligned}
\end{equation*}

In all the formulae for the Epstein zeta function, the various involved quantities are defined as follows.

Firstly, the involved quadratic forms and their derivatives are
\begin{align*}
\square \tilde{Q}_A =& \tilde Q_B - H_D\,\tilde Q_A,\quad \square^2 \tilde{Q}_A = \square\tilde Q_B- (\square H_D\,\tilde Q_A+H_D\,\square\tilde Q_A)\\
\square^3 \tilde{Q}_A =& \square^2\tilde Q_B- (\square^2 H_D\,\tilde Q_A+2\square H_D\,\square\,\tilde Q_A+H_D\,\square^2\tilde Q_A)\\
\square^4 \tilde{Q}_A =& \square^3\tilde Q_B- \Big(\square^3 H_D\,\tilde Q_A+3\square^2 H_D\square\tilde Q_A+3\square H_D\square^2\tilde Q_A+H_D\square^3\tilde Q_A\Big)\\
\square \tilde{Q}_B=&-H_D\,\tilde{Q}_B,\; \square^2 \tilde{Q}_B = (-\square\,H_D+H_D^2)\tilde{Q}_B,\;\square^3 \tilde{Q}_B = (-\square^2H_D+3H_D\square H_D-H_D^3)\tilde{Q}_B
\end{align*}
where
$$
\tilde Q_B(u,v) :=\frac{Q_B(u,v)}{\sqrt{D}},\quad Q_B(u,v) := Lu^2+2Muv+Nv^2.
$$
and other involved constants such as $H_D$ and their parametric derivatives are given by
\begin{align*}
\square C_{s}(D) &= -sH_D\cdot C_{s}(D),\quad C_{s+1}(D) =C_{s}(D)\cdot\pi/(s\,\sqrt{D}),\\
H_D &:= (GL+EN-2FM)/(2D) =: (\tilde G\tilde L+\tilde E\tilde N-2\tilde F\tilde M)/2,\\
K_D &:=\tilde{L}\tilde{N}-\tilde{M},\quad \square K_D = -2H_D\,K_D,\\
\square H_D &=K_D-2H_D^2,\quad \square^2 H_D =-2H_D\,K_D-4H_D\,\square\,H_D\\
\square^3 H_D &=(4H_D^2-2\square H_D)K_D-4(\square H_D)^2-4H_D\,\square^2H_D
\end{align*}
Note that if $L,M,N$ are the second fundamental form coefficients, then $H_D$ and $K_D$ become the mean and Gaussian curvatures.

\subsection{Vector-valued derivatives $\square^{(k)}Z(s)$ via the scalar derivatives $\square^{k}Z(s)$}

We show that the vector-valued $k$-th derivative $\square^{(k)}Z(s)$, as defined in (\ref{eq:square(2)},~\ref{eq:square(3)},~\ref{eq:square(1,4)}), can be expressed as a combination of scalar derivatives $\square^{k}Z(s)$, for $k=2,3,4$.

Firstly, the 5-component vector $\square^{(2)}Z(s)$ can be computed via 5 scalar second derivatives:
\begin{align*}
a &= \frac{\partial^2}{\partial E^2}Z_A,\quad b = \frac{1}{4}\frac{\partial^2}{\partial F^2}Z_A,\quad c = \frac{\partial^2}{\partial G^2}Z_A\\
d &= \Big(\frac{\partial}{\partial E}+\frac{1}{2}\frac{\partial}{\partial F}\Big)^2Z_A, \quad e = \Big(\frac{\partial}{\partial G}+\frac{1}{2}\frac{\partial}{\partial F}\Big)^2Z_A\\
\frac{1}{2}&\frac{\partial^2}{\partial E\partial F}Z_A = \frac{d-a-b}{2},\quad \frac{1}{2}\frac{\partial^2}{\partial F\partial G}Z_A = \frac{e-b-c}{2},
\end{align*}
so
$$\square^{(2)}Z_A = \left(a, \frac{d-a-b}{2}, b, \frac{e-b-c}{2}, c\right).$$
Secondly, the 7-component vector $\square^{(3)}Z(s)$ can be computed via 7 scalar third derivatives:
\begin{align*}
a &= \frac{\partial^3}{\partial E^3}Z_A,\quad b = \frac{1}{8}\frac{\partial^3}{\partial F^3}Z_A,\quad c = \frac{\partial^3}{\partial G^3}Z_A,\\
d &= \Big(\frac{\partial}{\partial E}+\frac{1}{2}\frac{\partial}{\partial F}\Big)^3Z_A, \quad e = \Big(\frac{\partial}{\partial E}-\frac{1}{2}\frac{\partial}{\partial F}\Big)^3Z_A,\\
f &= \Big(\frac{\partial}{\partial G}+\frac{1}{2}\frac{\partial}{\partial F}\Big)^3Z_A, \quad g = \Big(\frac{\partial}{\partial G}-\frac{1}{2}\frac{\partial}{\partial F}\Big)^3Z_A,\\
\frac{1}{4}\frac{\partial^3}{\partial E\partial F^2}Z_A &= \frac{d+e-2a}{6},\quad \frac{1}{2}\frac{\partial^3}{\partial E^2\partial F}Z_A = \frac{d-e-2b}{6},\\
\frac{1}{4}\frac{\partial^3}{\partial F^2\partial G}Z_A &= \frac{f+g-2c}{6},\quad \frac{1}{2}\frac{\partial^3}{\partial F\partial G^2}Z_A = \frac{f-g-2b}{6},
\end{align*}
so
$$\square^{(3)}Z_A = \left(a, \frac{d-e-2b}{6}, \frac{d+e-2a}{6}, b, \frac{f+g-2c}{6}, \frac{f-g-2b}{6}, c\right).$$
Lastly, to compute $\square^{(4)}$, we first define
$$
a_{p,q,r} :=\square^4\big|_{(L,M,N)=(p,q/2,r)} = \left(p\frac{\pd}{\pd E} +\frac{q}{2}\frac{\pd}{\pd F}+ r\frac{\pd}{\pd G}\right)^4
$$
and denote the 9 components of $\square^{(4)}$ as $\square^{(4)}_i$, $i=1,\dots,9$, then
$$
\square^{(4)}_{1,2,3,4,5} = M^{-1}  \mathbf{a}^{EF},\qquad \square^{(4)}_{9,8,7,6,5} = M^{-1} \mathbf{a}^{GF}
$$
where
$$
M^{-1} =
\begin{bmatrix}
 1                &  0                & 0                  & 0                   &  0\\
-\frac{1}{8}  & \frac{1}{4}   & -\frac{1}{24} & -\frac{1}{12}  &  \frac{1}{2}\\
-\frac{1}{6}  & \frac{1}{12} &                 0  & \frac{1}{12}   & -\frac{1}{6}\\
\frac{1}{8}   & -\frac{1}{8}  & \frac{1}{24}  & -\frac{1}{24}  & -\frac{1}{2}\\
0                 & 0                 & 0                  & 0                   & 1
\end{bmatrix}
,
\mathbf{a}^{EF} = 
\begin{bmatrix}
a_{1,0,0}\\a_{1,1,0}\\a_{1,2,0}\\a_{1,-1,0}\\a_{0,1,0},
\end{bmatrix}
,
\mathbf{a}^{GF} =
\begin{bmatrix}
a_{0,0,1}\\a_{0,1,1}\\a_{0,2,1}\\a_{0,-1,1}\\a_{0,1,0}.
\end{bmatrix}
$$

\section{Proof of the convergence rate of Wigner-type limits}
\label{sc:app:proof}
\begin{remark}
    The proofs of Theorem \ref{thm:wigner-type_converge_rate_1} and Lemma \ref{lemma:long_fourier_analysis} in this section are inspired by \cite[Theorem 3.1, Lemma 3.3]{marin2014corrected}, which are the one-dimensional analogs. As a result, our proofs complete the convergence analysis in \cite{marin2014corrected} for their two-dimensional quadrature.
\end{remark}

Throughout this section, $\eta(\uu)\in C_c^{\beta}([-a,a]^2)$ is a $\beta$-times continuously differentiable function that is compactly supported in $[-a,a]^2$, and satisfies
\begin{equation}
    \begin{aligned}
    \eta(\mathbf{0})=1,\; \eta(\uu) \equiv \eta(-\uu),\; \text{and}\; \frac{\partial^{i+j}}{\partial u^i\partial v^j}\eta(\mathbf{0}) = 0
    \end{aligned}
    \label{eq:eta_condition}
\end{equation}
for all non-negative integers $i$ and $j$ such that $0<i+j\leq 2K.$

We restate the definition \eqref{eq:unconverged_wigner-type} as follows:
\begin{equation}
\begin{aligned}
    \cL^m_h[f]:=& \lim_{N\to\infty}\sideset{}'\sum_{|\ii |<N} \frac{ f(\ii )}{Q(\ii )^{m+\frac{1}{2}}}\eta(\ii h) - \int\limits_{|\uu|<N+\frac{1}{2}}\frac{ f(\uu)}{Q(\uu)^{m+\frac{1}{2}}}\eta(\uu h)\,\d\uu\\
    =&\sideset{}'\sum_{|\ii |<\infty} \frac{ f(\ii )}{Q(\ii )^{m+\frac{1}{2}}}\eta(\ii h) - \int\limits_{|\uu|<\infty}\frac{ f(\uu)}{Q(\uu)^{m+\frac{1}{2}}}\eta(\uu h)\,\d\uu
\end{aligned}
\end{equation}
with
$$
Q(\uu) = Eu^2+2Fuv+Gv^2,\quad \text{ where }E,\,G,\,EG-F^2>0.
$$

\begin{theorem}
\label{thm:wigner-type_converge_rate}
Using the notations in Theorem \ref{thm:converged_wigner-type}, then for any integer $K \geq 1$ and  $m, d \geq 0$ such that
$$
\begin{aligned}
0&\leq m\leq 2K-1,\\
\frac{3m}{2}&\leq d\leq m+K-1,
\end{aligned}
$$
the estimate
    \begin{equation}
        \left|\cL^m_h\big[u^{2d-l}v^l\big]  - \cL^m\big[u^{2d-l}v^l\big]\right| = O(h^{2K}),\qquad l = 0,1,\dots, 2d
    \end{equation}
holds provided that $\eta(\uu)$ satisfies \eqref{eq:eta_condition} and is $4K$-times contiuously differentiable.
\end{theorem}
\begin{proof}
Let $s = 2(m-d)+1$ then $3-2K \leq s \leq 1$, then using Theorem \ref{thm:wigner-type_converge_rate_1} (to be proved below) with $\beta = 4K$ yields
$$
\cL_h^{m-d}[1] = Z(2(m-d)+1) + O(h^{2K}).
$$
Taking appropriate parametric derivatives on both sides according to Theorem \ref{thm:converged_wigner-type} gives
$$
\cL_h^{m}\big[u^{2d-l}v^l\big] = \cL^m\big[u^{2d-l}v^l\big] + O(h^{2K})
$$
where the order of the error term remains unchanged because of the following reason. 

Parametric differentiation in the form of Theorem \ref{thm:converged_wigner-type} only acts on the component $Q(\uu)^{-s/2}$ and is in such a way that does not change its homogeneous order in $\uu$, which is $\Theta(|\uu|^{-s})$. Therefore the analysis in Theorem \ref{thm:wigner-type_converge_rate_1} holds under parametric differentiation.
\qed\end{proof}

\begin{theorem} 
\label{thm:wigner-type_converge_rate_1}
If $s\in\RR, s\leq1<2K$ and $\eta$ satisfies \eqref{eq:eta_condition} and is at least $\beta$ times continuously differentiable such that $\beta + s \geq 2K+3$, $K\geq1$, then
$$\cL_h^{(s-1)/2}[1] = Z(s) + O(h^{2K})$$
\end{theorem}
\begin{proof}
To analyze the convergence property of
$$\cL_h^{(s-1)/2}[1] = \sideset{}'\sum_{|\ii |<\infty}\frac{\eta(\ii h)}{Q(\ii )^{\frac{s}{2}}} - \int_{|\uu|<\infty}\frac{\eta(\uu h)}{Q(\uu)^{\frac{s}{2}}}\,\d\uu,$$
we define a smooth cutoff function $\phi(\uu)$, such that $\phi(\uu) \equiv \phi(-\uu)$ and
$$
\phi(\uu) = \begin{cases}
0 & |\uu|\leq \frac{1}{2},\\
1 & |\uu|\geq 1.
\end{cases}
$$
and denote $f(\uu,h) := \eta(\uu h)\,Q(\uu)^{-\frac{s}{2}}$, then a partition of unity using $\phi(\uu)$ gives
$$
\begin{aligned}
\cL_h^{(s-1)/2}[1] =&\sideset{}'\sum_{|\ii |<\infty}f(\ii ,h)\phi(\ii ) - \int_{|\uu|<\infty}f(\uu,h)\phi(\uu)\,\d\uu\\
&+\sideset{}'\sum_{|\ii |<\infty}f(\ii ,h)(1-\phi(\ii )) - \int_{|\uu|<\infty}f(\uu,h)(1-\phi(\uu))\,\d\uu\\
=& \left(\sum_{|\ii |<\infty}f(\ii ,h)\phi(\ii ) - \int_{|\uu|<\infty}f(\uu,h)\phi(\uu)\,\d\uu\right) - \int_{|\uu|\leq1}f(\uu,h)(1-\phi(\uu))\,\d\uu\\
:=& I_1-I_2
\end{aligned}
$$
where the fact that $f(\mathbf{0},h)\phi(\mathbf{0}) = 0$ and $1-\phi(\ii)=0$ if $\ii\neq\mathbf{0}$ are used. We then estimate the quantities $I_1$ and $I_2$.
\begin{itemize}
\item We first bound $I_2$. Use $|\eta(\uu h)-1|\leq C' |\uu h|^{2K}$ for some constant $C'$, then
$$
\begin{aligned}
&\left| \int_{|\uu|\leq1}(\eta(\uu h)-1)\,Q(\uu)^{-\frac{s}{2}}(1-\phi(\uu))\,\d\uu\right|\\ &\leq  C' h^{2K} \int_{|\uu|\leq1}|\uu|^{2K}|Q(\uu)|^{-\frac{s}{2}} |1-\phi(\uu)|\d\uu\leq Ch^{2K}
\end{aligned}
$$
where we used the fact $|\uu|^{2K}|Q(\uu)|^{-s/2} = O(|\uu|^{2K-s})$ is integrable since $2K-s>0$. So, if we define
$$Z_2(s) = \int_{|\uu|\leq1}Q(\uu)^{-\frac{s}{2}}(1-\phi(\uu))\,\d\uu$$
which is independent of $h$ and integrable for $s\leq1$, then
$$
\begin{aligned}
|I_2 &- Z_2(s)| = \left| \int_{|\uu|\leq1}(\eta(\uu h)-1)\,Q(\uu)^{-\frac{s}{2}}(1-\phi(\uu))\,\d\uu\right| \leq Ch^{2K}
\end{aligned}
$$
\item For the summation and integral in $I_1$, since the function $f(\uu,h)$ is compactly supported, the Poisson summation formula (cf. \cite[Chapter 4]{pinsky2008introduction}) gives
\begin{equation}
    \sum_{|\ii |<\infty}f(\ii ,h)\phi(\ii ) - \int_{|\uu|<\infty}f(\uu,h)\phi(\uu)\,\d\uu = \sum_{|\kk|\neq0} \hat{f}_\phi(\kk,h)
\label{eq:poisson_sum}
\end{equation}
where $\kk = (k_1,k_2)\in\ZZ^2$,
$$\hat{f}_\phi(\kk,h) := \int_{|\uu|<\infty}f(\uu,h)\phi(\uu)e^{-2\pi i \kk\cdot \uu}\,\d\uu.$$
Using Lemma \ref{lemma:long_fourier_analysis} which we will prove next, there is a smooth function $Z_1(s)$ independent of $h$, such that
$$\left|\sum_{|\kk|\neq0} \hat{f}_\phi(\kk,h)-Z_1(s)\right| = O(h^{2K+1})$$
\end{itemize}
Combining all of above, we have
$$\left|\cL_h^{(s-1)/2}[1] - Z_1(s)+Z_2(s)\right| = O(h^{2K})$$
On the other hand, by Theorem \ref{thm:converged_wigner-type} $$\lim_{h\to0}\cL_h^{(s-1)/2}[1] = Z(s)$$ is analytic. Since both $Z_1$ and $Z_2$ are smooth and neither of them depend on $h$, one must have
$$Z_1(s)-Z_2(s)\equiv Z(s)$$
\qed\end{proof}

\begin{lemma}
\label{lemma:long_fourier_analysis}
Assume that $s\leq 1$ and $\eta$ satisfies \eqref{eq:eta_condition} and is at least $\beta$ times continuously differentiable such that $\beta + s \geq 2K+3$, $K\geq1$. Then for $\hat{f}_\phi(\kk,h)$ as defined in the Poisson summation formula \eqref{eq:poisson_sum}, there is a smooth function $Z_1(s)$ that does not depend on $h$, such that
$$\left|\sum_{|\kk|\neq0} \hat{f}_\phi(\kk,h) - Z_1(s)\right| = O(h^{2K+1})$$
\end{lemma}
Before proving the lemma, we define the following simplified notations to be used throughout the proof. Denote the function
$$S(\uu,s) = Q(\uu)^{-\frac{s}{2}}$$
and for $\kk=(k_1,k_2)$ and $\uu=(u,v)$ define the following
$$
\begin{aligned}
k &:= |\kk|_\infty = \max\{k_1,k_2\}\\
D^j &= D^j(\kk) :=
\begin{cases}
(\partial/\partial u)^j & \text{if }k_1\geq k_2\\
(\partial/\partial v)^j & \text{if }k_1 < k_2
\end{cases}\quad \text{for any }j\in\NN
\end{aligned}
$$
We also use the shorthand notation
$$g^{(j)}(\uu) \equiv D^jg(\uu)$$
to denote the derivatives of any smooth function $g$ when the associated $\kk$ is clear from context.
\begin{proof}
Since $\eta(\uu h)\phi(\uu)S(\uu,s)$ is compactly supported in $\frac{1}{2}\leq\uu\leq\frac{a}{h}$ and at least $\beta$-times continuously differentiable, then after integration-by-parts $\beta$ times we have
$$
I(h,\kk):=\int_{|\uu|<\infty}\eta(\uu h)\phi(\uu)S(\uu,s)e^{i\kk\cdot\uu}\d\uu =  \frac{i^{\beta}}{k^{\beta}}\int_{|\uu|<\infty}\big(D^{\beta }\eta(\uu h)\phi(\uu)S(\uu,s)\big)e^{i\kk\cdot\uu}\d\uu.
$$
Define
$$
W(\kk) := i^{\beta}\int_{|\uu|<\infty}\big(D^{\beta}\,\phi(\uu)S(\uu,s)\big)e^{i\kk\cdot\uu}\d\uu
$$
then, since $\phi^{(j)}(\uu)$ is compactly supported for $j\geq1$ and $|S^{(\beta)}(\uu,s)| = O(|\uu|^{-s-\beta})$ integrable for $\beta+s\geq2K+3\geq5$, we have
$$
|W(\kk)|\leq \int_{|\uu|<\infty}|\phi(\uu)S^{(\beta)}(\uu,s)|\,\d\uu + \sum_{j = 1}^{\beta} d_{\beta,j}\int_{|\uu|<\infty}|\phi^{(j)}(\uu)S^{(\beta-j)}(\uu,s)|\d\uu \leq C
$$
for some constant $C$, where $d_{\beta,j}$ are binomial coefficients. So $W(\kk)$ is independent of $h$ and uniformly bounded with the appropriate $\kk$ and $D^j$ correspondence as defined before the proof.

Next, define $c(\uu):=\phi(\uu)\eta(\uu h)$ and consider
\begin{equation}
\begin{aligned}
I(h,\kk) - \frac{W(\kk)}{k^{\beta}} =&  \frac{i^{\beta}}{k^{\beta}}\int_{|\uu|<\infty}\big(D^{\beta }(c(\uu)-\phi(\uu))S(\uu,s)\big)e^{i\kk\cdot\uu}\d\uu\\
=& \frac{i^{\beta}}{k^{\beta}}\sum_{j = 1}^{\beta} d_{\beta,j}\int_{|\uu|<\infty}\big(c^{(j)}(\uu)-\phi^{(j)}(\uu)\big)S^{(\beta-j)}(\uu,s)e^{i\kk\cdot\uu}\d\uu\\
\end{aligned}
\label{eq:I-W;lemma:long}
\end{equation}
Using the fact that $\phi^{(j)}(\uu) \equiv 0$ for $j > 0$ and $|\uu| > 1$, we have
$$
c^{(j)}(\uu) =
\begin{cases}
\sum_{\ell=0}^j d_{j,\ell}\,\phi^{(j-\ell)}(\uu)\,\eta^{(\ell)}(\uu h)\,h^\ell, & \frac{1}{2}\leq|\uu|\leq1\\
\eta^{(j)}(\uu h)\,h^j, &1<|\uu|\leq \frac{a}{h}
\end{cases}
$$
therefore, using the fact that $\phi(\uu)\equiv0$ for $|\uu|\leq \frac{1}{2}$, $\phi(\uu)\equiv1$ for $|\uu|\geq 1$ and $\eta(\uu h)\equiv0$ for $|\uu|\geq\frac{a}{L}$, the integrals in \eqref{eq:I-W;lemma:long} can be rewritten as follows
$$
\begin{aligned}
&\int_{|\uu|<\infty}\big(c^{(j)}(\uu)-\phi^{(j)}(\uu)\big)S^{(\beta-j)}(\uu,s)e^{i\kk\cdot\uu}\d\uu\\
=& \int_{\frac{1}{2}\leq|\uu|\leq1}\bigg(\sum_{\ell=0}^j d_{j,\ell}\,\phi^{(j-\ell)}(\uu)\,\eta^{(\ell)}(\uu h)\,h^\ell-\phi^{(j)}(\uu)\bigg)S^{(\beta-j)}(\uu,s)e^{i\kk\cdot\uu}\d\uu\\
  &+ \int_{1<|\uu|\leq \infty}\big(\eta^{(j)}(\uu h)\,h^j-\delta_{0,j}\big)S^{(\beta-j)}(\uu,s)e^{i\kk\cdot\uu}\d\uu\\
=& \sum_{\ell=0}^jd_{j,\ell}\,h^\ell\, \int_{\frac{1}{2}\leq|\uu|\leq1}\phi^{(j-\ell)}(\uu)\,\big(\eta^{(\ell)}(\uu h)-\delta_{0,\ell}\big)\,S^{(\beta-j)}(\uu,s)e^{i\kk\cdot\uu}\d\uu\\
  &+ h^j\int_{1<|\uu|\leq \frac{a}{h}}\big(\eta^{(j)}(\uu h)-\delta_{0,j}\big)S^{(\beta-j)}(\uu,s)e^{i\kk\cdot\uu}\d\uu\\
  &-\delta_{0,j}\int_{\frac{a}{h}<|\uu|<\infty}S^{(\beta)}(\uu,s)e^{i\kk\cdot\uu}\d\uu\\
:=& \sum_{\ell=0}^jd_{j,\ell}\, I_1^{(\ell)} + I_2 + I_3
\end{aligned}
$$
where $\delta_{i,j}$ is the Kronecker delta. Note that $\eta^{(\ell)}(\mathbf{0}) = 0$ for $\ell\leq 2K$, so by Taylor's theorem, and using the fact $\eta(\mathbf{0})=1$, we have
$$
|\eta^{(\ell)}(\uu h)-\delta_{0,\ell}| =  |\eta^{(\ell)}(\uu h)-\eta^{(\ell)}(\mathbf{0})| \leq C'\,|\uu h|^{2K+1-\ell}
$$
for some constant $C'$. Then we can estimate each of $I_1^{(\ell)},I_2,I_3$ as follows.
\begin{itemize}
\item For the first sum of integrals
$$
|I_1^{(\ell)}| \leq h^\ell\, (C' \,h^{2K+1-\ell})\int_{\frac{1}{2}\leq|\uu|\leq1}|\uu|^{2K+1-\ell}\,|S^{(\beta-j)}(\uu,s)|\d\uu \leq  C\,h^{2K+1}
$$
so
$$
\left| \sum_{\ell=0}^jd_{j,\ell}I_1^{(\ell)}\right|\leq \sum_{\ell=0}^jd_{j,\ell}|I_1^{(\ell)}|\leq C\,h^{2K+1}
$$
\item For the second integral
$$
\begin{aligned}
|I_2| \leq& h^j\,(C'\,h^{2K+1-j})\int_{1<|\uu|\leq \frac{a}{h}}|\uu|^{(2K+1-j)+(-s-\beta+j)}\d\uu\\
 =& C'\,h^{2K+1}\int_{1<|\uu|\leq \frac{a}{h}}|\uu|^{2K+1-s-\beta}\d\uu\\
 \leq& C (h^{2K+1} + h^{s+\beta-2})
\end{aligned}
$$
\item For the third integral
$$
|I_3| \leq \int_{\frac{a}{h}<|\uu|<\infty}|\uu|^{-s-\beta}\d\uu\leq C\,h^{s+\beta-2}
$$
\end{itemize}
Therefore, substituting all above estimates back into \eqref{eq:I-W;lemma:long} yields
$$
\begin{aligned}
\left|I(h,\kk) - \frac{W(\kk)}{k^{\beta}}\right| \leq& \frac{1}{k^{\beta}}\sum_{j = 1}^{\beta} d_{\beta,j}\left(\sum_{\ell=0}^jd_{j,\ell}|I_1^{(\ell)}|+|I_2|+|I_3|\right)\\
\leq&\, \frac{C}{k^\beta}(h^{2K+1}+h^{\beta+s-2})
\end{aligned}
$$
By assumption, $\beta+s \geq 2K+3$ and $s\leq1$, so substitute $\kk\mapsto 2\pi\kk$ in the above formula yields
$$
\left|\hat{f}_\phi(\kk,h) -  \frac{W(2\pi\kk)}{(2\pi k)^{\beta}}\right| \leq \frac{C}{k^{\beta}}\,h^{2K+1}\leq \frac{C}{k^{2(K+1)}}\,h^{2K+1}.
$$
Therefore, if we define
$$
Z_1(s):=\sum_{|\kk|\neq0}\frac{W(2\pi\kk)}{(2\pi k)^{\beta}}
$$
which converges since $\beta\geq 2K+3-s \geq4$ given $s\leq1$ and $K\geq1$, then
$$\left|\sum_{|\kk|\neq0} \hat{f}_\phi(\kk,h) - Z_1(s)\right| \leq C'\,h^{2K+1}\sum_{|\kk|\neq0}\frac{1}{k^{2(K+1)}}\leq C h^{2K+1}$$
\qed\end{proof}

\bibliographystyle{spmpsci}
\bibliography{bobi_quadr}

\end{document}